\documentclass[11pt,a4paper]{article}

\usepackage[T1]{fontenc}
\usepackage[utf8]{inputenc}
\usepackage{lmodern}
\usepackage{microtype}
\usepackage[colorlinks=true, linkcolor=blue]{hyperref}
\usepackage[margin=0.9in]{geometry}

\usepackage{amsmath,amssymb,amsthm}
\usepackage{mathtools}
\usepackage{mathrsfs}
\usepackage{stmaryrd}

\usepackage{graphicx}
\usepackage{enumitem}
\usepackage{mdframed}
\usepackage{authblk}
\usepackage[numbers,sort&compress]{natbib}

\graphicspath{{Fig/}}

\numberwithin{equation}{section}
\allowdisplaybreaks

\theoremstyle{plain}
\newtheorem{theorem}{Theorem}[section]
\newtheorem{lemma}[theorem]{Lemma}

\newtheorem{question}{Question}
\newtheorem{claim}{Claim}[section]

\theoremstyle{definition}

\theoremstyle{remark}

\surroundwithmdframed[
  linewidth=0.6pt,
  linecolor=black,
  backgroundcolor=white,
  innertopmargin=6pt,
  innerbottommargin=6pt,
  skipabove=10pt,
  skipbelow=10pt
]{claim}

\newcounter{condition}


\newenvironment{keywords}{%
  \par\medskip
  \noindent\textbf{Keywords. }\ignorespaces
}{\par}

\newcommand{\subjclass}[1]{%
  \par\smallskip
  \noindent\textbf{2020 Mathematics Subject Classification. }#1\par
}

\title{On hypercyclic and (common) upper frequently hypercyclic
subspaces}

\author[1]{Nacib G. Albuquerque}
\author[2]{Thiago R. Alves%
  \thanks{Corresponding author.
  Email: \href{mailto:alves.thiago@ufu.br}
  {alves.thiago@ufu.br}.}}
\author[2]{Geraldo Botelho}
\author[2]{Vin\'icius V. F\'avaro}

\affil[1]{%
  Departamento de Matem\'atica,
  Universidade Federal da Para\'iba,
  Jo\~ao Pessoa, PB 58051-900, Brazil
}

\affil[2]{%
  Instituto de Matem\'atica e Estat\'istica,
  Universidade Federal de Uberl\^andia,
  Uberl\^andia, MG 38400-902, Brazil
}

\date{}

\begin{document}

\maketitle

\begin{abstract}
Let $B$ be a unilateral weighted backward shift on $\ell_p$, $1 \leq p < \infty$, that admits a ${\mathscr{U}\text{\!-frequently}}$ hypercyclic subspace. We prove that $B$ admits such a subspace free of frequently hypercyclic vectors. The proof technique we develop also allows us to prove that $B$ admits a hypercyclic subspace free of ${\mathscr{U}\text{\!-frequently}}$ hypercyclic vectors, and to solve a question posed  by B\`es and Menet in 2015 on the existence of common ${\mathscr{U}\text{\!-frequently}}$ hypercyclic  subspaces.
\end{abstract}



\tableofcontents

\section{Introduction and Main Results}
\label{Sc-1}

Let \(X\) be an infinite-dimensional real or complex Banach space, and
let \(T\) be a bounded linear operator on \(X\). The \textit{orbit} of a vector
\(x\in X\) under \(T\) is the set
\[
\operatorname{Orb}(x,T):=\{x,Tx,T^2x,\ldots\}.
\]
If this set is dense in \(X\), then \(x\) is called a \textit{hypercyclic vector}
for \(T\). An operator which possesses such a vector is called
\textit{hypercyclic}.

We shall be concerned with two stronger forms of hypercyclicity:
frequent hypercyclicity, introduced by Bayart and Grivaux
\cite{BayGri2004,BayGri2006}, and \(\mathscr{U}\)-frequent
hypercyclicity, introduced by Shkarin \cite{Shkarin09}. In the former case, the set
of times at which the orbit visits any prescribed nonempty open subset
of \(X\) must have positive lower density; in the latter, positive upper density is required. The precise definitions are recalled in
Section~\ref{Sec-2}.

Hypercyclicity may also be strengthened in another direction. Instead
of asking for a single hypercyclic vector, one may ask for a closed
infinite-dimensional subspace all of whose nonzero vectors are
hypercyclic. Such a subspace \(M\) is called a \emph{hypercyclic
subspace} for \(T\colon X\to X\). The problem of determining which
operators admit hypercyclic subspaces has been extensively studied; see
\cite{BerMon1995}, the monographs \cite[Chap.~8]{BayMath-book} and
\cite[Chap.~10]{Gro-ErdMang2011-book}, and
\cite{Menet2014, ChanMad25, DasMun24, BayErnMen16, Bernal06,
GonLeonMont-2000, LeonMon97, LeonMul06, Montes96, Shkarin10,
BerCalLopPra25, Lopez24}.

The same idea may be applied to frequent and
\(\mathscr{U}\)-frequent hypercyclicity. A closed infinite-dimensional
subspace is called a \emph{frequently hypercyclic subspace}
(respectively, a \emph{\(\mathscr{U}\)-frequently hypercyclic
subspace}) for \(T\) if each of its nonzero vectors is frequently
hypercyclic (respectively, \(\mathscr{U}\)-frequently hypercyclic) for
\(T\). These notions were introduced, respectively, by Bonilla and
Grosse-Erdmann \cite{BonillaErdmann12} and by B\`es and Menet
\cite{BesMenet15}.

Weighted backward shifts provide a natural setting for the study of
these notions. Let \(\mathbf{w}=(\omega_k)_{k\geq1}\) be a bounded sequence of nonzero
scalars, called a \emph{weight sequence}. For \(1\leq p<\infty\), the
corresponding unilateral weighted backward shift on
\(\ell_p=\ell_p(\mathbb N_0)\), where
\(\mathbb N_0:=\{0,1,2,\ldots\}\), is defined by
\[
B_{\mathbf w}(x_0,x_1,\ldots)
=
(\omega_1x_1,\omega_2x_2,\ldots).
\]
We shall allow \(\ell_p\) to be either real or complex. Menet's
characterization of the weighted backward shifts which admit
hypercyclic subspaces will play an important role below
\cite{Menet2014}.

For unilateral weighted backward shifts, B\`es and Menet obtained a
complete characterization of the existence of a
\(\mathscr{U}\)-frequently hypercyclic subspace. Their result
\cite[Corollary~3.9]{BesMenet15} may be stated as follows.

\begin{theorem}
\label{thm-carac}
Let \(1\leq p<\infty\), and let
\(B_{\mathbf{w}}\colon\ell_p\to\ell_p\) be the unilateral weighted
backward shift associated with
\(\mathbf{w}=(\omega_k)_{k\geq1}\). The following assertions are
equivalent:
\begin{enumerate}[
    label=\textup{(\roman*)},
    ref=(\roman*),
    labelindent=1em,
    leftmargin=*,
    widest=iii,
    noitemsep,
    topsep=0pt
]
\item The weight sequence \(\mathbf{w}\) satisfies
\begin{align}
\label{eq-19-08-I}
\sum_{k\geq1}
|\omega_1\omega_2\cdots\omega_k|^{-p}
<\infty
\quad\text{and}\quad
\sup_{n\geq1}\inf_{k\geq0}
\prod_{j=1}^{n}|\omega_{k+j}|
\leq1.
\end{align}

\item
\label{thm-B}
\(B_{\mathbf{w}}\) admits a \(\mathscr{U}\)-frequently hypercyclic
subspace.

\item
\label{thm-C}
\(B_{\mathbf{w}}\) is \(\mathscr{U}\)-frequently hypercyclic and admits
a hypercyclic subspace.
\end{enumerate}
\end{theorem}

The existence assertions in \ref{thm-B} and \ref{thm-C} lead naturally
to the following finer questions.

\begin{question}
\label{q:A}
Which unilateral weighted backward shifts on \(\ell_p\) admit a
\(\mathscr{U}\)-frequently hypercyclic subspace containing no frequently
hypercyclic vectors?
\end{question}

\begin{question}
\label{q:B}
Which \(\mathscr{U}\)-frequently hypercyclic unilateral weighted
backward shifts on \(\ell_p\) admit a hypercyclic subspace containing no
\(\mathscr{U}\)-frequently hypercyclic vectors?
\end{question}

Menet \cite{Menet2015} constructed frequently hypercyclic weighted
backward shifts on \(\ell_p\) which admit a hypercyclic subspace but no
frequently hypercyclic subspace. Later, B\`es and Menet
\cite{BesMenet15} constructed frequently hypercyclic weighted backward
shifts which admit a \(\mathscr{U}\)-frequently hypercyclic subspace but
no frequently hypercyclic subspace. These results come close to
Questions~\ref{q:A} and~\ref{q:B}, but do not answer them. The point is that a subspace may contain vectors satisfying a stronger form of hypercyclicity without all of its nonzero vectors satisfying that
condition. Thus, the \(\mathscr{U}\)-frequently hypercyclic subspaces
constructed in these works may still contain frequently hypercyclic
vectors, and the hypercyclic subspaces may still contain
\(\mathscr{U}\)-frequently hypercyclic vectors. Questions~\ref{q:A}
and~\ref{q:B} require these vectors to be absent altogether.

The main construction of this paper does precisely this. For every
shift satisfying \eqref{eq-19-08-I}, it produces a
\(\mathscr{U}\)-frequently hypercyclic subspace containing no frequently
hypercyclic vectors. A variant of the construction produces a
hypercyclic subspace containing no \(\mathscr{U}\)-frequently
hypercyclic vectors. Together with Theorem~\ref{thm-carac}, these
constructions answer Questions~\ref{q:A} and~\ref{q:B} and yield the
following characterization.

\begin{theorem}
\label{thm-1}
Let \(1\leq p<\infty\), and let
\(B_{\mathbf{w}}\colon\ell_p\to\ell_p\) be the unilateral weighted
backward shift associated with
\(\mathbf{w}=(\omega_k)_{k\geq1}\). The following assertions are
equivalent:

\begin{enumerate}[
    label=\textup{(\roman*)},
    ref=(\roman*),
    labelindent=1em,
    leftmargin=*,
    widest=iii,
    noitemsep,
    topsep=0pt
]
\item
\label{thm-A-1}
The weight sequence \(\mathbf{w}\) satisfies \eqref{eq-19-08-I}.

\item
\label{thm-B-1}
\(B_{\mathbf{w}}\) admits a \(\mathscr{U}\)-frequently hypercyclic
subspace containing no frequently hypercyclic vectors.

\item
\label{thm-C-1}
\(B_{\mathbf{w}}\) is \(\mathscr{U}\)-frequently hypercyclic and admits
a hypercyclic subspace containing no \(\mathscr{U}\)-frequently
hypercyclic vectors.
\end{enumerate}
\end{theorem}

By Theorem~\ref{thm-carac}, the implications
\ref{thm-B-1}~\(\Rightarrow\)~\ref{thm-A-1} and
\ref{thm-C-1}~\(\Rightarrow\)~\ref{thm-A-1} are immediate. The proof of
Theorem~\ref{thm-1} is therefore reduced to the converse implications
\ref{thm-A-1}~\(\Rightarrow\)~\ref{thm-B-1} and
\ref{thm-A-1}~\(\Rightarrow\)~\ref{thm-C-1}.

Theorem~\ref{thm-1} has a geometric refinement. For every
\(\varepsilon>0\), the subspaces in \ref{thm-B-1} and
\ref{thm-C-1} can be chosen to be \((1+\varepsilon)\)-isomorphic to
\(\ell_p\) and still have the dynamical properties stated there. This
refinement could be deduced from the existence assertions and a result
of Pe{\l}czy\'nski \cite[Proposition~I.5.4]{guerre}; our construction
gives it directly.

The construction has a third application, concerning common
\(\mathscr{U}\)-frequently hypercyclic subspaces. Let
\(\{T_\gamma\colon X\to X\}_{\gamma\in\Gamma}\) be a family of
operators. A closed infinite-dimensional subspace \(M\) of \(X\) is
called a \emph{common \(\mathscr{U}\)-frequently hypercyclic subspace}
for this family if every \(x\in M\setminus\{0\}\) is
\(\mathscr{U}\)-frequently hypercyclic for every \(T_\gamma\),
\(\gamma\in\Gamma\).

For \(\mu>1\), put
\[
\mathbf{w}_{\mu}
=
\left(1+\frac{\mu}{n}\right)_{n\geq1}.
\]
For each fixed \(1\leq p<\infty\), consider the family
\[
\{B_{\mathbf{w}_{\mu}}\colon\ell_p\to\ell_p\}_{\mu>1}.
\]
Each member of this family admits a \(\mathscr{U}\)-frequently
hypercyclic subspace, and the family itself admits a common hypercyclic
subspace \cite[Corollaries~3.9 and~4.8]{BesMenet15}. B\`es and Menet
asked whether it also admits a common \(\mathscr{U}\)-frequently
hypercyclic subspace \cite[Problem~2]{BesMenet15}. Using a third variant of our construction, we obtain our second main result, which answers this question affirmatively and, in fact, in a stronger form.

\begin{theorem}
\label{thm-2}
Let \(1\leq p<\infty\). For each \(\mu>1\), set
\[
\mathbf{w}_{\mu}
=
\left(1+\frac{\mu}{n}\right)_{n\geq1},
\]
and let \(B_{\mathbf{w}_{\mu}}\colon\ell_p\to\ell_p\) be the associated
weighted backward shift. Then there exists a closed infinite-dimensional
subspace \(M\subseteq\ell_p\) such that, for every
\(x\in M\setminus\{0\}\) and every \(\mu>1\), the vector \(x\) is
\(\mathscr{U}\)-frequently hypercyclic but not frequently hypercyclic
for \(B_{\mathbf{w}_{\mu}}\).
\end{theorem}

The article is organized as follows. In Section~\ref{Sec-2}, we fix the
notation and recall the results needed later. The implications
\ref{thm-A-1}~\(\Rightarrow\)~\ref{thm-B-1} and
\ref{thm-A-1}~\(\Rightarrow\)~\ref{thm-C-1} are proved in
Sections~\ref{Sec-4} and~\ref{Sc-3}, respectively, and
Theorem~\ref{thm-2} is proved in Section~\ref{Sec-5}. The remaining
technical material is collected in the last two sections:
Section~\ref{Sc-6} contains the auxiliary lemmas, while
Section~\ref{section-technical-estimates} contains the proofs of the
technical estimates used in the main arguments. This arrangement keeps
the main constructions from being interrupted by their technical
verification.

\section{Background and Notation}
\label{Sec-2}

Throughout the article, $\mathbb{N}:=\{1,2,\ldots\}$, and
$\mathbb{K}$ denotes either $\mathbb{R}$ or $\mathbb{C}$.

\subsection{Frequent and
${\mathscr{U}\text{\!-frequent}}$ Hypercyclicity}

The \textit{lower asymptotic density} and
\textit{upper asymptotic density} of $A \subset \mathbb{N}_0$
are defined, respectively, by
\[
\underline{\mathrm{d}}(A)
:=
\liminf_{N \to \infty}
\dfrac{\#(A \cap [N])}{N+1}
\quad
\text{and}
\quad
\overline{\mathrm{d}}(A)
:=
\limsup_{N \to \infty}
\dfrac{\#(A \cap [N])}{N+1},
\]
where $[N] := \{0,1,\ldots,N\}$. Let $X$ be a Banach space
and let $T : X \to X$ be a bounded linear operator. The
\textit{return set} from $x$ to $E \subset X$ under $T$ is
defined by
\[
\mathcal{N}_T(x,E)
:=
\{n \in \mathbb{N}_0 : T^n x \in E\}.
\]
A vector $x \in X$ is called \textit{frequently hypercyclic}
(resp. \textit{${\mathscr{U}\textit{\!-frequently}}$
hypercyclic}) for $T$ if
\[
\underline{\mathrm{d}}(\mathcal{N}_T(x,U)) > 0
\quad
\left(
\text{resp. }
\overline{\mathrm{d}}(\mathcal{N}_T(x,U)) > 0
\right)
\]
for every nonempty open subset $U$ of $X$. The operator $T$ is called \textit{frequently hypercyclic}
(resp. \textit{${\mathscr{U}\textit{\!-frequently}}$
hypercyclic}) if it admits a frequently hypercyclic
(resp. ${\mathscr{U}\text{\!-frequently}}$ hypercyclic)
vector. We denote by $\mathrm{HC}(T)$, $\mathrm{UFHC}(T)$, and
$\mathrm{FHC}(T)$ the sets of hypercyclic,
${\mathscr{U}\text{\!-frequently}}$ hypercyclic, and
frequently hypercyclic vectors for $T$, respectively.

For each $k \in \mathbb{N}_0$, let
$\pi_k :
\mathbb{K}^{\mathbb{N}_0}
\to \mathbb{K}$
be given by
\[
\pi_k((z_n)_{n \geq 0}) = z_k.
\]
By $(e_k)_{k \geq 0}$ we denote the canonical
unit vectors of $\ell_p$. The following elementary lemma provides two criteria that we
shall use to show that a vector is not
$\mathscr{U}$-frequently hypercyclic or not frequently
hypercyclic.

\begin{lemma}
\label{lem-1-subs-upper}
Let $S$ be a bounded linear operator on $\ell_p$,
$1 \leq p < \infty$, and let $x_0 \in \ell_p$. Then:
\begin{enumerate}[
    label=\textup{(\alph*)},
    ref=(\alph*),
    labelindent=1em,
    leftmargin=*,
    widest=iii,
    noitemsep,
    topsep=0pt
]
\item
    \label{item-a-15-06-1929}
    If there exists $k \in \mathbb{N}_0$ such that
    \begin{align}
    \label{eq-aux-1-subs-upper}
    \underline{\mathrm{d}}
    \big(
        \{n \in \mathbb{N} :
        \pi_k(S^n x_0)=0\}
    \big)
    =1,
    \end{align}
    then $x_0$ is not a
    ${\mathscr{U}\text{\!-frequently}}$ hypercyclic vector
    for $S$.

    \item
    \label{item-b-15-06-1929}
    If there exists $k \in \mathbb{N}_0$ such that
    \begin{align}
    \label{eq-31-07-A}
    \overline{\mathrm{d}}
    \big(
        \{n \in \mathbb{N} :
        \pi_k(S^n x_0)=0\}
    \big)
    =1,
    \end{align}
    then $x_0$ is not a frequently hypercyclic vector
    for $S$.
\end{enumerate}
\end{lemma}
\begin{proof}
\ref{item-a-15-06-1929}
Let $B(e_k,1)$ denote the open ball in $\ell_p$ centered at
$e_k$ with radius $1$, and let
\[
A:=\{x\in\ell_p:\pi_k(x)=0\}.
\]
Then $B(e_k,1)\subset A^c$, and hence, by assumption,
\begin{align*}
\overline{\mathrm{d}}
\big(
    \{n\in\mathbb{N}:S^nx_0\in B(e_k,1)\}
\big)
&\leq
\overline{\mathrm{d}}
\big(
    \{n\in\mathbb{N}:S^nx_0\in A^c\}
\big)
\\&=
1-
\underline{\mathrm{d}}
\big(
    \{n\in\mathbb{N}:S^nx_0\in A\}
\big)
\\&=0,
\end{align*}
so that $x_0$ is not $\mathscr{U}$-frequently hypercyclic
for $S$.

\ref{item-b-15-06-1929}
It suffices to swap $\overline{\mathrm{d}}$ and
$\underline{\mathrm{d}}$ in the proof of
\ref{item-a-15-06-1929}.
\end{proof}

\subsection{Miscellaneous Notation}
Given a weight sequence ${\bf w}=(\omega_k)_{k\geq 1}$, let
$S_{\bf w}$ denote the linear map on $c_{00}(\mathbb{N}_0)$ defined by
\[
S_{\bf w}e_k
=
\omega_{k+1}^{-1}e_{k+1},
\qquad
k\in\mathbb{N}_0.
\]
Since $B_{\bf w}S_{\bf w}=I$ on
$c_{00}(\mathbb{N}_0)$,
for every $i,j,k\in\mathbb{N}_0$, we have
\begin{align}
\label{eq-11-06-1408}
\begin{split}
B_{\bf w}^iS_{\bf w}^je_k
=
\left\{
\begin{array}{c l}
(\omega_{k+1}\omega_{k+2}\cdots\omega_{k+j-i})^{-1}
e_{k+j-i}
&
\mbox{if } i<j;
\\[2mm]
e_k
&
\mbox{if } i=j;
\\[2mm]
\omega_k\omega_{k-1}\cdots\omega_{k-i+j+1}
\,e_{k-i+j}
&
\mbox{if } j<i\leq j+k;
\\[2mm]
0
&
\mbox{if } i>j+k.
\end{array}
\right.
\end{split}
\end{align}

For each $k \in \mathbb{N}_0$ and each real
number $p \geq 1$, we set
\begin{align*}
\mathscr{D}_{k}^{\mathbb{K},p}
:=
\bigg\{
(\alpha_0,\ldots,\alpha_k)
\in \mathbb{Q}_{\mathbb{K}}^{k+1}
:
\sum_{j=0}^k|\alpha_j|^p<1,
\ \alpha_k\neq0
\bigg\},
\end{align*}
where
$\mathbb{Q}_{\mathbb{C}}:=\mathbb{Q}+i\mathbb{Q}$
and
$\mathbb{Q}_{\mathbb{R}}:=\mathbb{Q}$.
We will, without further notice, regard
$\mathbb{K}^{k+1}$ as a subspace of
$\mathbb{K}^{\mathbb{N}_0}$ in certain situations by
identifying
\[
(\alpha_0,\ldots,\alpha_k)
\quad\text{with}\quad
(\alpha_0,\ldots,\alpha_k,0,0,\ldots).
\]
It is evident that
\[
\overline{
    \bigcup_{k\geq0}
    \mathscr{D}_{k}^{\mathbb{K},p}
}^{\ell_p}
=
\overline{B_{\ell_p}},
\]
where $B_{\ell_p}$ denotes the open unit ball of $\ell_p$.

Finally, for $m,n\in\mathbb{N}_0$ with $m\leq n$, we will
denote
\[
\llbracket m,n\rrbracket
:=
\{k\in\mathbb{N}_0:m\leq k\leq n\}.
\]

\section{Proof of \ref{thm-A-1}$\Rightarrow$\ref{thm-B-1} in Theorem~\ref{thm-1}} \label{Sec-4}

\subsection{Construction of a family of intervals depending on $\mathbf{w}$} \label{Sub-sec-4}
From Lemma~\ref{lemma-pesos}, there exists $\beta \geq 1$ such that
\begin{align} \label{eq-06-08-B}
\mathcal{A}^{\mathbf{w}}_{N, \beta} := \bigg\{k \in \mathbb{N} : \max_{k < u \leq k+N} \prod_{j = u}^{k+N} |\omega_j|  \leq \beta \bigg\} \ \ \text{is infinite for every } N \in \mathbb{N}.
\end{align}
By the first condition in \eqref{eq-19-08-I}, we may fix a strictly increasing function
$\theta:\mathbb N_0\to \mathbb N$ satisfying
\begin{align} \label{eq-11-06-1534}
\sum_{\ell \geq \theta(k)}
\frac{{\|{\bf w}\|}_\infty^{kp}}{|\omega_1 \omega_2 \cdots \omega_\ell|^p} < 2^{-k} \ \ \mbox{for each} \ \ k \in \mathbb{N}_0.
\end{align}
Note that the summability condition in \eqref{eq-19-08-I} implies \(\|\mathbf w\|_\infty>1\). Moreover, fix a map $\tau:2\mathbb N_0+1\to\mathbb N_0$ such that
$\tau^{-1}(n)$ is infinite for every $n\in\mathbb N_0$.

We will now construct a sequence $(j_k)_{k\geq 0}$ in $\mathbb{N}$ and intervals
$$
I_k(u) = \llbracket a_k(u),b_k(u) \rrbracket \subset \mathbb{N},
$$
where $k \geq 0$ and $u = 1, \ldots, j_k$, with the following properties:

If $k\ge 1$,
\begin{align}
\frac{a_k(1)-b_{k-1}(j_{k-1})}{a_k(1)}
    \ge \frac{k}{k+1}.
    \label{cond:1a}
\end{align}
For $k\ge 2$ even,
\begin{align}
j_k &= 1,
    \label{cond:2a}
\\
a_k(1)-1
    &\in \mathcal{A}^{\mathbf w}_{b_{k-1}(j_{k-1}),\beta},
    \label{cond:2b}
\\
b_k(1)-a_k(1)+1
    &= b_{k-1}(j_{k-1}).
    \label{cond:2c}
\end{align}
If $k\ge 1$ is odd, then
\begin{align}
a_k(1)-b_{k-1}(1)
    &\ge \theta\bigl(\tau(k)+b_{k-1}(1)\bigr),
    \label{cond:3a}
\\
\frac{j_k}
     {a_k(1)+(j_k-1)\theta(\tau(k))+1}
    &\ge
    \frac{1}{2\theta(\tau(k))},
    \label{cond:3b}
\\
a_k(u)-a_k(u-1)
    &= \theta(\tau(k)) \mbox{ for } 1<u\le j_k,
    \label{cond:3c}
\\
b_k(u)-a_k(u)
    &= \tau(k) \mbox{ for } 1\le u\le j_k.
    \label{cond:3d}
\end{align}

The construction is carried out recursively. Start with
\[
j_0 =1, \, a_0(1) = b_0(1) = 1,
\]
hence $I_0(1)=\{1\}$. Let $k \geq 1$. Find $a_k(1)$ satisfying
\eqref{cond:1a}. If $k \geq 2$ is even, also require $a_k(1)$
to satisfy \eqref{cond:2b} by using \eqref{eq-06-08-B}, and then
find $b_k(1)$ satisfying \eqref{cond:2c}. The construction in the
even case is completed by setting $j_k=1$. If $k \geq 1$ is odd,
also require $a_k(1)$ to satisfy \eqref{cond:3a}. Then find $j_k$
satisfying \eqref{cond:3b}, followed by $a_k(u)$ ($u>1$) and
$b_k(u)$ satisfying \eqref{cond:3c} and \eqref{cond:3d}.

For future reference, note that conditions \eqref{cond:2b} and
\eqref{cond:2c} imply that
\begin{align} \label{eq-even-product-control}
\max_{a_k(1) \leq u \leq b_k(1)}
\prod_{\ell=u}^{b_k(1)}|\omega_{\ell}|
\leq\beta \quad \mbox{for every} \, \, k \geq 2 \, \, \mbox{even}.
\end{align}

\subsection{Construction of a $(1+\varepsilon)$-isomorphic embedding
$T:\ell_p\to\ell_p$}
\label{Subsec-3.2}

Fix $\varepsilon>0$. We shall construct a sequence
$(x_q)_{q\geq 0}\subset\ell_p$ such that the formula
\begin{align}
\label{def-of-T}
\begin{split}
T:\ell_p&\longrightarrow\ell_p,\\
\lambda=(\lambda_q)_{q\geq0}
&\longmapsto
T(\lambda):=\sum_{q\geq0}\lambda_qx_q
\end{split}
\end{align}
defines a linear $(1+\varepsilon)$-isomorphism onto its range.

For each $n\in\mathbb{N}_0$, let $(A_{n,q})_{q\geq0}$ be a partition of
$\tau^{-1}(n)$ into infinite subsets. For each
$(n,q)\in\mathbb N_0 \times\mathbb N_0$, fix a bijection
\[
\Phi_{n,q}^{p}:A_{n,q}\longrightarrow
\mathscr D_n^{\mathbb K,p}.
\]
For every $q\in\mathbb N_0$, define 
\begin{align}
\label{eq-sec2-1}
x_q
:=
e_{b_{2q}(1)}+\varepsilon\widetilde x_q,
\qquad
\widetilde x_q
:=
\sum_{n\geq 0}\sum_{r\in A_{n,q}}
\sum_{u=1}^{j_r}
S_{\mathbf w}^{a_r(u)}
\bigl(\Phi_{n,q}^{p}(r)\bigr).
\end{align}

We first prove that $\widetilde x_q$ is well defined. The construction
in Subsection~\ref{Sub-sec-4} gives
\[
a_{r+1}(1)>b_r(j_r)
\qquad (r\geq0).
\]
Moreover, for every \(r\in2\mathbb N_0+1\) and every
\(1\leq u<j_r\), we have
\[
a_r(u+1)-a_r(u)
=
\theta(\tau(r))
>
\tau(r)
=
b_r(u)-a_r(u).
\]
Consequently, the family
\[
\bigl\{
\llbracket a_r(u),b_r(u)\rrbracket:
r\in\mathbb N_0,\ 1\leq u\leq j_r
\bigr\}
\]
is pairwise disjoint.

If $r\in A_{n,q}$, then $\tau(r)=n$, and hence
\begin{align*}
b_r(u)-a_r(u)=n.
\end{align*}
Since
\begin{align} \label{eq-20-08-0856}
\Phi_{n,q}^{p}(r)\in
\operatorname{span}\{e_0,\ldots,e_{n}\}, \ \ \ n \geq 0,
\end{align}
it follows that
\[
\operatorname{supp}
\left(
S_{\mathbf w}^{a_r(u)}
\bigl(\Phi_{n,q}^{p}(r)\bigr)
\right)
\subseteq
\llbracket a_r(u),b_r(u)\rrbracket.
\]
Thus the series defining $\widetilde x_q$ is a coordinatewise
well-defined sum of finitely supported vectors with pairwise disjoint
supports.

For $n\in\mathbb N_0$, let $P_n$ denote the canonical projection onto
$\operatorname{span}\{e_0,\ldots,e_{n}\}$.
Therefore,
\begin{align} \label{eq-20-08-0850}
\bigl\|S_{\mathbf w}^{a}\circ P_n\bigr\|^p
\leq
\sum_{s=1}^{n+1}
\frac{\|\mathbf w\|_\infty^{np}}
{\displaystyle\prod_{i=1}^{a+s-1}|\omega_i|^p}.
\end{align}
Fix $q\in\mathbb N_0$. For every $n\in\mathbb N_0$, set
\[
r_n:=\min A_{n,q}
\qquad\text{and}\qquad
c_n:=b_{r_n-1}(1).
\]
Since the supports occurring in \eqref{eq-sec2-1} are pairwise
disjoint, and since
$\|\Phi_{n,q}^{p}(r)\|_p<1$, we obtain
\begin{align} \label{eq-20-08-1719}
\begin{split}
\|\widetilde x_q\|_p^p
&\overset{\eqref{eq-sec2-1}}{=}
\sum_{n\geq 0}\sum_{r\in A_{n,q}}\sum_{u=1}^{j_r}
\left\|
S_{\mathbf w}^{a_r(u)}
\bigl(\Phi_{n,q}^{p}(r)\bigr)
\right\|_p^p \\
 &\overset{\eqref{eq-20-08-0856}}{=} \sum_{n\geq 0}\sum_{r\in A_{n,q}}\sum_{u=1}^{j_r}
\left\|
(S_{\mathbf w}^{a_r(u)} \circ P_n)
\bigl(\Phi_{n,q}^{p}(r)\bigr)
\right\|_p^p
\\
&\overset{\eqref{eq-20-08-0850}}{\leq}
\sum_{n \geq 0}\sum_{r\in A_{n,q}}\sum_{u=1}^{j_r}
\sum_{s=1}^{n+1}
\frac{\|\mathbf w\|_\infty^{np}}
{\displaystyle\prod_{i=1}^{a_r(u)+s-1}|\omega_i|^p}
\\&\hspace{0.1in}\leq
\sum_{n \geq 0}
\sum_{k\geq a_{r_n}(1)}
\frac{\|\mathbf w\|_\infty^{np}}
{|\omega_1\cdots\omega_k|^p},
\end{split}
\end{align}
where the last inequality follows from the fact that, for each fixed $n$, the integers
\[
a_r(u)+s-1,
\qquad
r\in A_{n,q},\quad
1\leq u\leq j_r,\quad
1\leq s\leq n+1,
\]
are pairwise distinct and the smallest possible value among them is
$a_{r_n}(1)$.

Since $r_n\in\tau^{-1}(n)$, condition \eqref{cond:3a} gives
\[
a_{r_n}(1) \geq a_{r_n}(1)-c_n
\geq
\theta(n+c_n).
\]
Using also $\|\mathbf w\|_\infty\geq1$ and the defining property of
$\theta$, we can continue \eqref{eq-20-08-1719} as follows:
\begin{align*}
\|\widetilde x_q\|_p^p
\leq
\sum_{n\geq 0}
\sum_{k\geq\theta(n+c_n)}
\frac{\|\mathbf w\|_\infty^{(n+c_n)p}}
{|\omega_1\cdots\omega_k|^p}
\leq
\sum_{n\geq 0}2^{-(n+c_n)}
\leq
\sum_{n\geq1}2^{-n}
=1.
\end{align*}
Hence $\widetilde x_q\in\overline B_{\ell_p}$ and
$x_q\in\ell_p$, with \(\|x_q\|_p\leq1+\varepsilon.\)

We next observe that the supports of the vectors \((x_q)_{q\geq0}\)
are pairwise disjoint. Recall that
\[
A_{n,q}\subseteq\tau^{-1}(n).
\]
Thus, if \(n\neq n'\), then
\[
A_{n,q}\cap A_{n',q'}=\emptyset,
\]
since distinct fibers of \(\tau\) are disjoint. If \(n=n'\) and
\(q\neq q'\), the same conclusion follows from the pairwise
disjointness of the sets \(A_{n,q}\) as \(q\) varies. Consequently,
\[
\left(\bigcup_{n \geq 0}A_{n,q}\right)
\cap
\left(\bigcup_{n \geq 0}A_{n,q'}\right)
=\emptyset
\qquad(q\neq q').
\]
It follows from the definition of \(x_q\) that
\[
\operatorname{supp}(x_q)
\subseteq
\{b_{2q}(1)\}
\cup
\bigcup_{n \geq 0}
\bigcup_{r\in A_{n,q}}
\bigcup_{u=1}^{j_r}
\llbracket a_r(u),b_r(u)\rrbracket.
\]
Note that, when
\(q\neq q'\), all the intervals appearing in the corresponding
support inclusions for \(x_q\) and \(x_{q'}\) have distinct indices.
The pairwise disjointness of all the intervals constructed in
Subsection~\ref{Sub-sec-4} now yields
\[
\operatorname{supp}(x_q)\cap\operatorname{supp}(x_{q'})
=\emptyset
\qquad(q\neq q').
\]
It also gives
\begin{align} \label{eq-24-08-1153}
\pi_{b_{2q}(1)}(\widetilde x_q)=0
\qquad\text{and hence}\qquad
\pi_{b_{2q}(1)}(x_q)=1.
\end{align}
Let $\lambda=(\lambda_q)_{q\geq0}\in\ell_p$. For
$0\leq m\leq M$, the disjointness of the supports gives
\begin{align*}
\left\|
\sum_{q=m}^{M}\lambda_qx_q
\right\|_p^p
=
\sum_{q=m}^{M}|\lambda_q|^p\|x_q\|_p^p
\leq
(1+\varepsilon)^p
\sum_{q=m}^{M}|\lambda_q|^p \underset{M,m \to \infty}{\longrightarrow} 0.
\end{align*}
Thus the series in \eqref{def-of-T} converges in $\ell_p$, and
\begin{align}
\label{eq-aux-4}
\|T(\lambda)\|_p^p
=
\sum_{q\geq0}|\lambda_q|^p\|x_q\|_p^p
\leq
(1+\varepsilon)^p\|\lambda\|_p^p.
\end{align}
Consequently, $T$ is linear and continuous, with
$\|T\|\leq1+\varepsilon$.

Finally,
\begin{align}
\|T(\lambda)\|_p^p
=
\sum_{q\geq0}|\lambda_q|^p\|x_q\|_p^p
\nonumber
\geq
\sum_{q\geq0}
|\lambda_q|^p
\left|
\pi_{b_{2q}(1)}(x_q)
\right|^p
\overset{\eqref{eq-24-08-1153}}{=}
\|\lambda\|_p^p.
\end{align}
From this and \eqref{eq-aux-4}, it follows that $T$ is a
$(1+\varepsilon)$-isomorphism onto its range.

\subsection{The range of $T$ lies (up to the zero vector) in
$\mathrm{UFHC}(B_{\mathbf w})$}
\label{subsection-3.3}

We now prove that
\[
T(\ell_p)\setminus\{0\}\subset\mathrm{UFHC}(B_{\mathbf w}).
\]
Fix \(x\in T(\ell_p)\setminus\{0\}\). Then \(x=T(\lambda)\) for some
nonzero \(\lambda=(\lambda_q)_{q\geq0}\in\ell_p\). Put
\begin{align}
x=\sum_{q\geq q_0}\lambda_qx_q,
 \qquad q_0:=\min\{q\in\mathbb N_0:\lambda_q\neq0\}.
\label{eq-sec2-C}
\end{align}
Since upper frequent hypercyclicity is invariant under multiplication
by a nonzero scalar, we may assume that
\begin{align}
\lambda_{q_0}=1.
\label{eq-25-07-A}
\end{align} 

We use a version of the Bourdon--Feldman theorem for upper frequent
hypercyclicity: a vector \(x\in\ell_p\) is upper frequently hypercyclic
for \(B_{\mathbf w}\) if there exists a nonempty open subset \(O\) of
\(\ell_p\) such that
\(
\overline{\mathrm d}\big(
\mathcal N_{B_{\mathbf w}}(x,W)
\big)>0
\)
for every nonempty open set \(W\subset O\). The proof follows the same
lines as that of \cite[Theorem 6.27]{BayMath-book}, with lower density
replaced by upper density.

Let \(W\subset\varepsilon B_{\ell_p}\) be a nonempty open subset of
\(\ell_p\). Choose \(\delta>0\) and \(y_0\in W\) such that \(B(y_0,\delta)\subset W.\) Note that, to complete the proof, it suffices to show that there is $\mathcal{S} \subset \mathbb{N}$ such that
\begin{enumerate}[label=\textup{(\alph*)}]
\item
\label{item-30-08-2226}
 there is \(s_0\in\mathbb N_0\) such that
\(B_{\mathbf w}^{s'}x\in B(y_0,\delta)\) whenever
\(s'\in\mathcal S\cap(s_0,\infty)\);

\item
\label{item-30-08-2227}
 \(\overline{\mathrm d}(\mathcal S)>0\).
\end{enumerate}
We proceed in three steps.

\smallskip
\noindent\emph{Step 1: Construction of $\mathcal{S}$.} Choose \(n_0\in\mathbb N\) such that
\begin{align}
2^{-n_0+1}
<
\min\left\{
\frac{\delta^p}{4^p\varepsilon^p},
\frac{\delta^p}
{\varepsilon^p2^{p+1}\|\lambda\|_p^p}
\right\}
=:\eta_0.
\label{eq-25-07-E}
\end{align}
Since
\( \varepsilon^{-1}B(y_0,\delta/4)\subset B_{\ell_p}\) and \(\bigcup_{n\geq 0}\mathscr D_n^{\mathbb K,p}\) is dense in
\(B_{\ell_p}\), there exist \(n^*\geq n_0\) and a sequence
\((y_k)_{k\geq1}\) of distinct vectors such that
\begin{align}
y_k\in
\varepsilon^{-1}B(y_0,\delta/4)
\cap\mathscr D_{n^*}^{\mathbb K,p}
\qquad (k\geq1).
\label{29-07-C}
\end{align}
Indeed, one may first choose a rational finitely supported vector
\(z\in \varepsilon^{-1}B(y_0,\delta/4)\). One then chooses an index
\(n^*\geq n_0\), larger than every index in \(\operatorname{supp}(z)\),
and adds a sufficiently small nonzero rational coordinate at \(n^*\).
Varying the value of this additional coordinate slightly produces the
required sequence.

Since
\[
\Phi_{n^*,q_0}^p:
A_{n^*,q_0}\longrightarrow
\mathscr D_{n^*}^{\mathbb K,p}
\] 
is a bijection, after reordering \((y_k)_{k\geq1}\), if necessary,
we obtain
\begin{align}
\Phi_{n^*,q_0}^p(r_k)=y_k \qquad \text{with} \quad r_1<r_2<r_3<\cdots
\text{ in }A_{n^*,q_0}.
\label{29-07-E}
\end{align}
Define
\begin{align}
\mathcal S
:=
\left\{
a_{r_k}(u):
k\geq1,\ 1\leq u\leq j_{r_k}
\right\}.
\label{def-of-set-S}
\end{align}

\smallskip
\noindent\emph{Step 2: Proof of \ref{item-30-08-2226}.} Choose \(l_0>q_0\) so large that
\begin{align} \label{eq-05-09-26-1542}
\sum_{q\geq l_0}|\lambda_q|^p
<
\frac{\delta^p}{2^{p+1}\beta^p},
\end{align}
where \(\beta\) is the constant fixed in the construction of
Subsection~\ref{Sub-sec-4}. Since there
are only finitely many \(q<l_0\), we may then choose
\(l_1\geq l_0\) such that
\begin{align}
B_{\mathbf w}^{l_1}
\big(e_{b_{2q}(1)}\big)=0
\qquad\text{whenever }q<l_0.
\label{eq-07-08-G}
\end{align} 
Since \(\mathcal S\) is unbounded, we may choose
\(k_0\geq1\) and \(1\leq u_0\leq j_{r_{k_0}}\) such that
\begin{align}
s_0:=a_{r_{k_0}}(u_0)\geq l_1
\qquad\text{and}\qquad
2^{-s_0}<\eta_0.
\label{eq-30-07-B}
\end{align}

For each $s \geq 1$, set
\begin{equation}
\label{eq-return-errors}
\begin{aligned}
E_1^{(3)}(s)
&:=
\big\|
\varepsilon B_{\mathbf w}^{s}\widetilde{x}_{q_0}-y_0
\big\|_p,\\
E_2^{(3)}(s)
&:=
\sum_{q\geq q_0}|\lambda_q|^p
\big\|
B_{\mathbf w}^{s}e_{b_{2q}(1)}
\big\|_p^p,\\
E_3^{(3)}(s)
&:=
\varepsilon^p
\sum_{q>q_0}|\lambda_q|^p
\big\|
B_{\mathbf w}^{s}\widetilde{x}_q
\big\|_p^p.
\end{aligned}
\end{equation}
Using the definition of \(x_q\), the equality
\(\lambda_{q_0}=1\), and the pairwise disjointness of the relevant
supports, we obtain
\begin{align}
\big\|B_{\mathbf w}^{s}x-y_0\big\|_p
\leq
E_1^{(3)}(s)+\big(E_2^{(3)}(s)+E_3^{(3)}(s)\big)^{1/p}.
\label{eq-10-08-D}
\end{align}

\begin{claim}
\label{claim-return-time-estimates}
For every \(s'\in\mathcal S\) with \(s'>s_0\), we have
\[
E_1^{(3)}(s')<\frac{\delta}{2},
\qquad
E_2^{(3)}(s')<\frac{\delta^p}{2^{p+1}},
\qquad
E_3^{(3)}(s')<\frac{\delta^p}{2^{p+1}}.
\]
\end{claim}

The proof of Claim~\ref{claim-return-time-estimates} is postponed to
Section~\ref{section-technical-estimates}. Combining the claim with
\eqref{eq-10-08-D}, we obtain
\begin{align}
\big\|B_{\mathbf w}^{s'}x-y_0\big\|_p
<
\frac{\delta}{2}
+
\left(
\frac{\delta^p}{2^{p+1}}
+
\frac{\delta^p}{2^{p+1}}
\right)^{1/p}
\nonumber =
\frac{\delta}{2}
+
\left(\frac{\delta^p}{2^p}\right)^{1/p}
=\delta.
\label{eq-11-08-C}
\end{align}

\smallskip
\noindent\emph{Step 3: Proof of \ref{item-30-08-2227}.} Since
\(r_k\in A_{n^*,q_0}\subset\tau^{-1}(n^*)\), we have
\begin{align*}
\frac{
\big|
\mathcal S\cap
\llbracket 0,a_{r_k}(j_{r_k})\rrbracket
\big|
}{
a_{r_k}(j_{r_k})+1
}
&\geq
\frac{
\big|
\{a_{r_k}(u):1\leq u\leq j_{r_k}\}
\big|
}{
a_{r_k}(j_{r_k})+1
}
\\&\hspace{-0.05in}\overset{\eqref{cond:3c}}{=}
\frac{j_{r_k}}
{
a_{r_k}(1)
+(j_{r_k}-1)\theta(n^*)+1
}
\\&\hspace{-0.05in}\overset{\eqref{cond:3b}}{\geq}
\frac{1}{2\theta(n^*)}.
\end{align*}
Since \(a_{r_k}(j_{r_k})\to\infty\) as \(k \to \infty\), it follows that
\[
\overline{\mathrm d}(\mathcal S)
\geq
\limsup_{k\to\infty}
\frac{
\big|
\mathcal S\cap
\llbracket 0,a_{r_k}(j_{r_k})\rrbracket
\big|
}{
a_{r_k}(j_{r_k})+1
}
\geq
\frac{1}{2\theta(n^*)}
>0.
\]

\subsection{The range of $T$ does not intersect
$\mathrm{FHC}(B_{\mathbf w})$}
\label{Subsection-3.4}

Fix \(x\in T(\ell_p)\setminus\{0\}\), and define
\[
\mathcal M
:=
\bigcup_{r\geq0}
\bigcup_{u=1}^{j_r}
\llbracket a_r(u),b_r(u)\rrbracket.
\] 
By the definition of the vectors \(x_q\), \(\operatorname{supp}(x)\subseteq\mathcal M.\) Thus, $\pi_0(B_{\mathbf w}^n x)=0$ whenever $n\notin\mathcal M$. By Lemma~\ref{lem-1-subs-upper}(b), it suffices to prove that
\(
\overline{\mathrm d}(\mathbb N_0\setminus\mathcal M)=1.\) Since the ordering of the intervals gives
\[
\llbracket
b_{r-1}(j_{r-1})+1,\,
a_r(1)-1
\rrbracket
\subseteq
\mathbb N_0\setminus\mathcal M, \qquad r \geq 1,
\]
we obtain
\begin{align}
\frac{
\big|
(\mathbb N_0\setminus\mathcal M)
\cap
\llbracket0,a_r(1)-1\rrbracket
\big|
}{
a_r(1)
}
\geq
\frac{
a_r(1)-b_{r-1}(j_{r-1})-1
}{
a_r(1)
}
\nonumber \overset{\eqref{cond:1a}}{\geq}
\frac{r}{r+1}
-
\frac{1}{a_r(1)} \underset{r \to \infty}{\longrightarrow} 1,
\label{eq-09-08-A}
\end{align}
 as desired.

\section{Proof of \ref{thm-A-1}$\Rightarrow$\ref{thm-C-1} in
Theorem~\ref{thm-1}}
\label{Sc-3}

By \cite[Theorem~4]{BayRuz2015}, the first condition in
\eqref{eq-19-08-I} implies that $B_{\mathbf w}$ is
$\mathscr{U}$-frequently hypercyclic. It remains to construct a
hypercyclic subspace for $B_{\mathbf w}$ containing no
$\mathscr{U}$-frequently hypercyclic vectors. The construction is a
variant of the one used in Section~\ref{Sec-4}. In view of the
similarities between the two constructions, we present the proofs in
this section more concisely.

\subsection{Construction of the intervals}
\label{Subsec-4.1}

Retain the constant $\beta$ and the maps $\theta$ and $\tau$ fixed in
Subsection~\ref{Sub-sec-4}. We construct a new sequence of intervals
\[
J_k=\llbracket c_k,d_k\rrbracket\subset\mathbb N,
\qquad k\geq0,
\]
with the following properties. First, for every $k\geq1$,
\begin{align}
\frac{c_k-d_{k-1}}{d_k+1}
\geq
\frac{k}{k+1}.
\label{cond:4-gap}
\end{align}
If $k\geq2$ is even, then
\begin{align}
c_k-1
&\in
\mathcal A^{\mathbf w}_{d_{k-1},\beta},
\label{cond:4-even-a}
\\
d_k-c_k+1
&=d_{k-1}.
\label{cond:4-even-b}
\end{align}
If $k\geq1$ is odd, then
\begin{align}
c_k-d_{k-1}
&\geq
\theta\bigl(\tau(k)+d_{k-1}\bigr),
\label{cond:4-odd-a}
\\
d_k-c_k
&=\tau(k).
\label{cond:4-odd-b}
\end{align}

The intervals are obtained recursively. Set
\[
c_0=d_0=1.
\]
Suppose that $J_0,\ldots,J_{k-1}$ have been chosen. If $k$ is odd,
choose $c_k$ sufficiently large that \eqref{cond:4-gap} and
\eqref{cond:4-odd-a} hold after setting
$d_k=c_k+\tau(k)$. If $k$ is even, choose
\[
c_k-1\in\mathcal A^{\mathbf w}_{d_{k-1},\beta}
\]
sufficiently large that \eqref{cond:4-gap} holds after setting
$d_k=c_k+d_{k-1}-1$. This is possible because
$\mathcal A^{\mathbf w}_{d_{k-1},\beta}$ is infinite. The construction
also gives
\[
d_{k-1}<c_k
\qquad(k\geq1),
\]
so the intervals $(J_k)_{k\geq0}$ are pairwise disjoint and occur in
their natural order.

For later use, conditions \eqref{cond:4-even-a} and
\eqref{cond:4-even-b} give
\begin{align}
\max_{c_k\leq u\leq d_k}
\prod_{\ell=u}^{d_k}|\omega_\ell|
\leq\beta
\qquad(k\geq2\text{ even}).
\label{eq:4-even-products}
\end{align}

\subsection{Construction of the embedding}
\label{Subsec-4.2} 

Fix $\varepsilon>0$, and retain the partitions $(A_{n,q})_{q\geq0}$
of $\tau^{-1}(n)$ and the bijections
\[
\Phi_{n,q}^{p}:A_{n,q}\longrightarrow
\mathscr D_n^{\mathbb K,p}
\]
fixed in Subsection~\ref{Subsec-3.2}. For $q\in\mathbb N_0$, set
\begin{align}
x_q
:=e_{d_{2q}}+\varepsilon\widetilde x_q, \quad
\widetilde x_q
:=
\sum_{n\geq0}
\sum_{r\in A_{n,q}}
S_{\mathbf w}^{c_r}
\bigl(\Phi_{n,q}^{p}(r)\bigr).
\label{eq:4-xq}
\end{align}
If $r\in A_{n,q}$, then $r$ is odd, $\tau(r)=n$, and
\eqref{cond:4-odd-b} gives
\[
\operatorname{supp}
\left(
S_{\mathbf w}^{c_r}
\bigl(\Phi_{n,q}^{p}(r)\bigr)
\right)
\subseteq J_r.
\]

The proof of
Subsection~\ref{Subsec-3.2} now applies with only notational changes.
Indeed, if
\[
r_n:=\min A_{n,q}
\qquad\text{and}\qquad
h_n:=d_{r_n-1},
\]
then condition \eqref{cond:4-odd-a} gives
\[
c_{r_n}\geq\theta(n+h_n),
\]
which is precisely the estimate used there to prove that
$\widetilde x_q\in\overline B_{\ell_p}$. The present argument is
slightly simpler, since there is only one interval $J_r$ for each
$r$. The supports of the vectors $(x_q)_{q\geq0}$ are again pairwise
disjoint, and
\[
\pi_{d_{2q}}(\widetilde x_q)=0
\qquad\text{and}\qquad
\pi_{d_{2q}}(x_q)=1.
\]
Hence, the formula
\begin{align*}
T:\ell_p&\longrightarrow\ell_p,
\nonumber\\
(\lambda_q)_{q\geq0}
&\longmapsto
\sum_{q\geq0}\lambda_qx_q
\end{align*}
defines a linear operator satisfying
\[
\|\lambda\|_p
\leq
\|T(\lambda)\|_p
\leq
(1+\varepsilon)\|\lambda\|_p
\qquad(\lambda\in\ell_p).
\]
Thus $T$ is a $(1+\varepsilon)$-isomorphism onto its range, and
$T(\ell_p)$ is closed and infinite-dimensional. 

\subsection{The nonzero vectors in the range are hypercyclic}
\label{Subsec-4.3}

Fix $x\in T(\ell_p)\setminus\{0\}$, and write
\[
x=T(\lambda)
=
\sum_{q\geq q_0}\lambda_qx_q,
\qquad
q_0:=\min\{q\in\mathbb N_0:\lambda_q\neq0\}.
\]
Since hypercyclicity is invariant under multiplication by a nonzero scalar, we may assume that $\lambda_{q_0}=1.$

We shall prove that the closure of $\operatorname{Orb}(x,B_{\mathbf w})$
contains $\varepsilon B_{\ell_p}$. The Bourdon--Feldman theorem
\cite[Theorem~3.13]{BayMath-book} then implies that $x$ is hypercyclic
for $B_{\mathbf w}$. Choose $y_0\in \varepsilon B_{\ell_p}$ and $\delta>0$ such that
\(B(y_0,\delta)\subset \varepsilon B_{\ell_p}.\) As in Subsection~\ref{subsection-3.3}, there are
$n^*\in\mathbb N_0$ and distinct vectors $(y_k)_{k\geq1}$ such that
\[
y_k\in
\varepsilon^{-1}B(y_0,\delta/4)
\cap
\mathscr D_{n^*}^{\mathbb K,p}
\qquad(k\geq1).
\]
After reordering this sequence, we obtain
\[
\Phi_{n^*,q_0}^{p}(r_k)=y_k
\qquad \text{with } \, r_1<r_2<\cdots
  \text{ in }A_{n^*,q_0}.
\]
For each $k$, set \(s_k:=c_{r_k}\) and define
\begin{align} 
\begin{split} \label{eq:4-return-errors}
E_1^{(4)}(k)
&:=
\bigl\|
\varepsilon B_{\mathbf w}^{s_k}\widetilde x_{q_0}-y_0
\bigr\|_p,
\\
E_2^{(4)}(k)
&:=
\sum_{q\geq q_0}|\lambda_q|^p
\bigl\|B_{\mathbf w}^{s_k}e_{d_{2q}}\bigr\|_p^p,
\\
E_3^{(4)}(k)
&:=
\varepsilon^p
\sum_{q>q_0}|\lambda_q|^p
\bigl\|B_{\mathbf w}^{s_k}\widetilde x_q\bigr\|_p^p.
\end{split}
\end{align}
The pairwise disjointness of the relevant supports gives
\begin{align}
\|B_{\mathbf w}^{s_k}x-y_0\|_p
\leq
E_1^{(4)}(k)
+
\bigl(E_2^{(4)}(k)+E_3^{(4)}(k)\bigr)^{1/p}.
\label{eq:4-return-decomposition}
\end{align}

\begin{claim}
\label{claim-sec4-return-time-estimates}
There exists $k_0\geq1$ such that, for every $k\geq k_0$,
\[
E_1^{(4)}(k)<\frac{\delta}{2},
\qquad
E_2^{(4)}(k)<\frac{\delta^p}{2^{p+1}},
\qquad
E_3^{(4)}(k)<\frac{\delta^p}{2^{p+1}}.
\]
\end{claim}

The proof of Claim~\ref{claim-sec4-return-time-estimates} is postponed
to Section~\ref{section-technical-estimates}. Combining the claim with
\eqref{eq:4-return-decomposition}, we obtain, for every $k\geq k_0$,
\begin{align*}
\|B_{\mathbf w}^{s_k}x-y_0\|_p
&<
\frac{\delta}{2}
+
\left(
\frac{\delta^p}{2^{p+1}}
+
\frac{\delta^p}{2^{p+1}}
\right)^{1/p}
=\delta.
\end{align*}
As $B(y_0,\delta)$ 
was an arbitrary ball in $\varepsilon B_{\ell_p}$, the closure of the orbit of $x$ contains
$\varepsilon B_{\ell_p}$, as desired.

\subsection{The range contains no
$\mathscr{U}$-frequently hypercyclic vectors}
\label{Subsec-4.4}

Let
\[
\mathcal M:=\bigcup_{k\geq0}J_k.
\]
For every $x\in T(\ell_p)$, the construction gives
$\operatorname{supp}(x)\subseteq\mathcal M$. Therefore,
\[
\pi_0(B_{\mathbf w}^n x)=0
\quad \mbox{whenever } \, n\in\mathbb N_0\setminus\mathcal M.
\] 
By Lemma~\ref{lem-1-subs-upper}(a), it is enough to prove that
\begin{align} \label{eq-25-08-1447}
\underline{\mathrm d}
(\mathbb N_0\setminus\mathcal M)=1.
\end{align}

For $k\geq1$, we have
\[
\{0\}
\cup
\llbracket d_{k-1}+1,c_k-1\rrbracket
\subseteq
\mathbb N_0\setminus\mathcal M.
\] 
It follows from \eqref{cond:4-gap} that
\begin{align}
\frac{
\big|
(\mathbb N_0\setminus\mathcal M)
\cap\llbracket0,d_k\rrbracket
\big|
}{d_k+1}
&\geq
\frac{c_k-d_{k-1}}{d_k+1}
\geq
\frac{k}{k+1}.
\label{eq:4-density-endpoints}
\end{align}
The quotient
\[
\frac{
\big|
(\mathbb N_0\setminus\mathcal M)
\cap\llbracket0,N\rrbracket
\big|
}{N+1}
\]
decreases as $N$ runs through $J_k$ and increases as $N$ runs through
the gap between $J_k$ and $J_{k+1}$. Hence, for
$c_k\leq N<c_{k+1}$, it is bounded below by the quotient in
\eqref{eq:4-density-endpoints}. Since $\frac{k}{k+1}\to1$ as $k\to\infty$, \eqref{eq-25-08-1447} follows.

\section{Proof of Theorem \ref{thm-2}} \label{Sec-5}

\subsection{Construction of the coordinate and parameter intervals}
\label{Subsec-5.1}

The main goal of this subsection is to construct coordinate and parameter intervals satisfying certain prescribed properties. We first introduce the necessary notation and describe the intervals to be constructed. We then state their required properties and, finally, carry out the construction.

\subsubsection{Notation}
Given \(L,m,n\in\mathbb N\), with \(m\geq2\), put
\begin{flalign}
\hspace{0.1in}&\bullet \quad 
\sigma_n:=\dfrac{n+1}{n};\nonumber\\
&\bullet \quad C_{m,L} \geq 2 \text{ an integer  such that } &&\nonumber\\ &\hspace{1.8in}\sum_{d\geq (C_{m,L})^{k}}
\frac{[2(m+2)]^{L \cdot p}}
{\displaystyle
 \prod_{\ell=1}^{d}
 \left(1+\frac{\sigma_m}{\ell}\right)^p}
<2^{-k}
\qquad(k\geq2);
\label{eq:5-between-family-tail}\\
&\bullet \quad \Theta_L \geq L \text{ an integer such that } \nonumber&&\\
& \hspace{1.8in}\sum_{d \geq \Theta_L}
 \frac{\|\mathbf w_{\sigma_L}\|_\infty^{L \cdot p}}
 {\displaystyle
  \prod_{\ell=1}^{d}
  \left(1+\frac{\sigma_L}{\ell}\right)^p}
 <2^{-L};
 \label{eq:5-within-family-tail}\\
 &\bullet \quad \mathfrak R
 :=
 \left\{
 (q,n,y,m,h):
 q,n\in\mathbb N_0,
 \, \, y\in\mathscr D_n^{\mathbb K,p},
 \, \,  m,h\in\mathbb N,\, \, m\geq2
 \right\}; \nonumber&&\\
 &\bullet \quad (\mathcal R_s)_{s\geq1} \text{ a sequence in which every member of }
\mathfrak R \text{ occurs infinitely often.} \nonumber
\end{flalign}
By Lemma~\ref{lem:exponential-tail}, the choice of $C_{m,L}$ is feasible. The choices of $\Theta_L$ and $(\mathcal R_s)_{s\geq1}$  are possible since $\sigma_L \cdot p > 1$ and $\mathfrak R$ is countable, respectively.

\subsubsection{Coordinate and parameter intervals} \label{subsubsec-1}

The coordinate and parameter intervals (which will be constructed below) are as follows.
\begin{flalign}
&\bullet \quad \text{ For each } s \geq 1, \text{ we have } \nonumber\\ &\hspace{1in}\mathcal{B}_s :=  \rrbracket P_{s-1},P_s\rrbracket \cap 2\mathbb{N}\text{ with } 0 = P_0 < P_1 < P_2 < \cdots \text{ are even numbers.}\nonumber\\
\hspace{0.1in}&\bullet \quad \text{For every } k\geq1, \text{ we have an integer }
N_k\geq 1 \text{ and intervals }&&\nonumber\\
&\hspace{1.8in}\widetilde I_k(u)
:=
\llbracket\widetilde a_k(u),\widetilde b_k(u)\rrbracket,
\qquad 1\leq u\leq N_k.
\label{eq-12-08-R} \\
&\bullet \quad \text{For every even } k,\text{ we have}&&\nonumber\\
&\hspace{1in}(q_k,n_k,y_k,m_k,h_k)\in\mathfrak R, 
\qquad
r_k\in\mathbb Q_{>1},
\qquad
J_k\subset\mathbb R_{>1}. \label{eq-30-09-26-0420}&&
\end{flalign}

\subsubsection{Required properties} \label{subsubsec-2}

The coordinate and parameter intervals satisfy the following properties.

$\bullet$ \quad For every \(k\geq2\),
\begin{align}
\frac{
\widetilde a_k(1)-\widetilde b_{k-1}(N_{k-1})
}{
\widetilde a_k(1)
}
\geq\frac{k-1}{k}.
\label{cond:5-gap}
\end{align}

$\bullet$ \quad For every even \(k\),
\begin{align}
\widetilde a_k(u+1)-\widetilde a_k(u)
&=\Theta_{n_k+1}
\quad \mbox{for} \, \, 1\leq u<N_k,
\label{cond:5-even-spacing}
\\
\frac{N_k}{\widetilde a_k(N_k)+1}
&\geq\frac{1}{2\Theta_{n_k+1}},
\label{cond:5-even-density}
\\
\widetilde b_k(u)-\widetilde a_k(u)
&=n_k
\quad \mbox{for} \, \, 1\leq u\leq N_k,
\label{cond:5-even-length}
\\
\sum_{d\geq\widetilde a_k(1)}
\frac{[2(m_k+2)]^{(n_k+1)p}}
{\displaystyle
\prod_{\ell=1}^{d}
\left(1+\frac{\sigma_{m_k}}{\ell}\right)^p}
&<2^{-k},
\label{cond:5-even-tail}
\\
\sigma_{m_k}<r_k&<m_k+1,
\label{eq:5-local-cover}
\\
\left\|
B_{\mathbf w_\mu}^{\widetilde a_k(u)}
S_{\mathbf w_{r_k}}^{\widetilde a_k(u)}y_k-y_k
\right\|_p
&<2^{-h_k}
\quad \mbox{for} \, \, \mu\in J_k,\ 1\leq u\leq N_k.
\label{eq:5-parameter-control}
\end{align}
Moreover, for any \(\mathcal R_s=(q,n,y,m,h)\), we have
\begin{align}
[\sigma_m,m]
\subseteq\bigcup_{k\in\mathcal B_s}J_k
\label{cond:5-parameter-cover} \, \, \mbox{ and } \, \, (q_k,n_k,y_k,m_k,h_k)
=\mathcal{R}_s
\quad \mbox{for} \, \, k\in\mathcal B_s.
\end{align}

$\bullet$ \quad For $k$ odd,
\begin{align}
N_k=1 \ \ \text{and} \ \ \left\|
 B_{\mathbf w_\mu}^{j}e_{\widetilde b_k(1)}
\right\|_p^p
\leq2
\quad \mbox{for} \, \,
\left(
\begin{array}{c}
k\geq3\ \text{odd},\ 1<\mu\leq k,\\
1\leq j\leq\widetilde b_{k-1}(N_{k-1})
\end{array}
\right).
\label{eq:5-odd-product-control}
\end{align}

\subsubsection{The construction}

We construct the required coordinate and parameter intervals recursively, in stages indexed by \(s\). First put
\[
P_0:=0,\qquad N_1:=1,\qquad
\widetilde a_1(1):=\widetilde b_1(1):=1.
\]
At the beginning of stage \(s\), write
\[
\mathcal R_s=(q,n,y,m,h),
\]
set \(k:=P_{s-1}+2\) and \(\tau_k:=\sigma_m\), and proceed as follows. The endpoint \(P_s\) will be defined when this stage terminates. Assign
\[
(q_k,n_k,y_k,m_k,h_k):=(q,n,y,m,h)
\]
and set
\[
\widetilde a_k(1)
:=
k\widetilde b_{k-1}(N_{k-1})
+(C_{m,n+1})^{\,k}+n+1.
\]
This choice gives \eqref{cond:5-gap} and, by
\eqref{eq:5-between-family-tail}, also gives \eqref{cond:5-even-tail}.  Now put
\[
\begin{aligned}
N_k:=\widetilde a_k(1), \quad
\widetilde a_k(u)
:=\widetilde a_k(1)+(u-1)\Theta_{n+1}, \, \mbox{ and } \, \widetilde b_k(u)
:=\widetilde a_k(u)+n \, \text{ for } 
1\leq u\leq N_k.
\end{aligned}
\]
Then \eqref{cond:5-even-spacing} and
\eqref{cond:5-even-length} are immediate.  Moreover,
\[
\widetilde a_k(N_k)+1
= N_k + (N_k-1)\Theta_{n+1} +1 \leq 2\Theta_{n+1}N_k,
\]
which gives \eqref{cond:5-even-density}.

Put
\begin{align}
\eta_k:=\min\left\{m-\tau_k,\frac{\Delta_k}{4}\right\} \quad \mbox{ with } \, \, \Delta_k
:=
\dfrac{2^{-h}}
{\sum_{\ell=1}^{\widetilde{b}_k(N_k)} \frac 1\ell}.
\label{eq:5-parameter-radius}
\end{align}
Choose
\[
r_k\in\mathbb Q\cap
\left(\tau_k+\eta_k,\tau_k+2\eta_k\right)
\]
and set
\[
J_k:=(r_k-\Delta_k,r_k]\cap\mathbb R_{>1}.
\]
Since \(\eta_k\leq\Delta_k/4\),
\begin{align}
[\tau_k,\tau_k+\eta_k]\subseteq J_k.
\label{eq:5-cover-step}
\end{align}
We also have \[\sigma_m \leq \tau_k < r_k < \tau_k +  \dfrac{\Delta_k}{2} \leq \tau_k +1 \leq m+1,\] and hence
\eqref{eq:5-local-cover} holds.

If \(\mu\in J_k\), then \(0\leq r_k-\mu<\Delta_k\).  Thus, for
\(1\leq t\leq n+1\) and \(1\leq u\leq N_k\),
\begin{align*}
0
\leq
1-
\prod_{\ell=t}^{\widetilde a_k(u)+t-1}
\frac{\ell+\mu}{\ell+r_k}
\leq
(r_k-\mu)
\sum_{\ell=t}^{\widetilde a_k(u)+t-1}
\frac1{\ell+r_k}
<
\Delta_k
\sum_{\ell=1}^{\widetilde b_k(N_k)}\frac1\ell
=2^{-h}.
\end{align*}
Since \(\|y\|_p<1\), this gives
\eqref{eq:5-parameter-control}.

We next construct the odd coordinate interval with index $k+1.$ Set
\begin{align} \label{eq-01-10-26-0609}
N_{k+1}&:=1, \, \, 
\widetilde a_{k+1}(1)
:=2\lceil p\rceil(k+1)\widetilde b_k(N_k), \, \,
\widetilde b_{k+1}(1)
:=\widetilde a_{k+1}(1)+\widetilde b_k(N_k)-1.
\end{align}
This gives \eqref{cond:5-gap} for \(k+1\).  Moreover, using
\(\log(1+x)\leq x\) for \(x\geq0\), we obtain
\begin{align*}
p\sum_{\ell=\widetilde a_{k+1}(1)}^{\widetilde b_{k+1}(1)}
\log\left(1+\frac{k+1}{\ell}\right)
&\leq
p(k+1)
\sum_{\ell=\widetilde a_{k+1}(1)}^{\widetilde b_{k+1}(1)}
\frac1\ell \overset{\eqref{eq-01-10-26-0609}}{\leq}
\frac{p(k+1)\widetilde b_k(N_k)}
{\widetilde a_{k+1}(1)}
\\
&\hspace{-0.1in}\overset{\eqref{eq-01-10-26-0609}}{=}
\frac{p}{2\lceil p\rceil}
\leq
\frac12
<
\log 2.
\end{align*} 
Consequently, if \(1<\mu\leq k+1\) and
\(1\leq j\leq\widetilde b_k(N_k)\), then
\begin{align*}
\left\|
B_{\mathbf w_\mu}^{j}e_{\widetilde b_{k+1}(1)}
\right\|_p^p
&\leq
\prod_{\ell=\widetilde a_{k+1}(1)}^{\widetilde b_{k+1}(1)}
\left(1+\frac{k+1}{\ell}\right)^p
<2.
\end{align*}
This proves \eqref{eq:5-odd-product-control}.

Depending on the value of \(\tau_{k+2}:=\tau_k+\eta_k\), we proceed in one of two ways:

\begin{enumerate}[label=(\Roman*)]
\item If \(\tau_{k+2}<m\), we keep \(s\) fixed, replace \((k,\tau_k)\) by \((k+2,\tau_{k+2})\), and repeat the construction with the same quintuple \((q,n,y,m,h)\).

\item If \(\tau_{k+2}=m\), set
\[
P_s:=k,\qquad
\mathcal B_s:=\rrbracket P_{s-1},P_s\rrbracket\cap2\mathbb N.
\]
Thus \(\mathcal B_s\) contains all even indices used during stage \(s\), including the terminal index \(k\). We then replace \(s\) by \(s+1\) and \(k\) by \(k+2\), write
\[
\mathcal R_s=(q,n,y,m,h),
\]
reset \(\tau_k:=\sigma_m\), and repeat the construction.
\end{enumerate}

We need to show that, for each fixed $s$, case (I) cannot repeat indefinitely. This follows from Lemma~\ref{lem:finite-completion}, which guarantees that case (II) is reached after finitely many steps. Consequently, each stage terminates and $\mathcal B_s$ is finite. Since the construction starts with $s=1,$ replaces $s$ only by $s+1,$ and uses the even integers successively, the sets $(\mathcal B_s)_{s\geq1}$ are successive and form a partition of $2\mathbb N.$ The stated assignments follow directly from the recursive steps. Moreover, at the beginning of each stage, $\tau_k$ is initialized at $\sigma_m.$ The inclusions \eqref{eq:5-cover-step} cover each successive increment, and the terminal step reaches $m.$ Hence \eqref{cond:5-parameter-cover} follows. This completes the construction.

For later use, we record an immediate consequence of the
construction. By \eqref{cond:5-gap}, \eqref{cond:5-even-spacing},
\eqref{cond:5-even-length}, and
\(\Theta_{n_k+1}\geq n_k+1\), the coordinate intervals are pairwise
disjoint and occur in their natural order.


\subsection{Construction of a $(1+\varepsilon)$-isomorphic embedding
$T:\ell_p\to\ell_p$}
\label{Subsec-5.2}

We keep the notation and choices made in
Subsection~\ref{Subsec-5.1}. In particular, for each even \(k\),
\(r_k\in\mathbb Q_{>1}\) is the parameter chosen in that construction.
Moreover, for each even $k,$ there is a unique $s \geq 1$ such that $k \in \mathcal{B}_s,$ and if \(
\mathcal R_s=(q,n,y,m,h),
\) then
\[
(q_k,n_k,y_k,m_k,h_k)=(q,n,y,m,h).
\]
Thus, \(y_k\in\mathscr D_{n_k}^{\mathbb K,p}\).

Fix \(\varepsilon>0\). For every \(q\in\mathbb N_0\), we set
\begin{align} \label{eq-sec2-1-open}
x_q :=
e_{\widetilde b_{2q+1}(1)}
+\varepsilon\widetilde x_q, \qquad \widetilde x_q
:=
\sum_{\substack{k\in2\mathbb N\\q_k=q}}
\sum_{u=1}^{N_k}
S_{\mathbf w_{r_k}}^{\widetilde a_k(u)}(y_k).
\end{align}

We first show that the series in \eqref{eq-sec2-1-open} is well defined
in \(\mathbb K^{\mathbb N_0}\). For every even \(k\), write
\[
y_k=(z_{k,0},\ldots,z_{k,n_k}).
\]
Since \(y_k\in\mathscr D_{n_k}^{\mathbb K,p}\) and, by
\eqref{cond:5-even-length}, the cardinality of
\(\widetilde I_k(u)\) is \(n_k+1\), we have
\begin{align}
\operatorname{supp}\left(
S_{\mathbf w_{r_k}}^{\widetilde a_k(u)}(y_k)
\right)
\subseteq
\widetilde I_k(u).
\label{eq-10-08-A}
\end{align}
A straightforward calculation using \eqref{cond:5-gap},
\eqref{cond:5-even-spacing}, \eqref{cond:5-even-length}, and
\(\Theta_{n_k+1}\geq n_k+1\) shows that the intervals
\(\widetilde I_k(u)\) are pairwise disjoint. Hence, by
\eqref{eq-10-08-A}, the supports of the summands in
\eqref{eq-sec2-1-open} are pairwise disjoint. Thus, the series in
\eqref{eq-sec2-1-open} is well defined in
\(\mathbb K^{\mathbb N_0}\). We next prove that it converges in
\(\ell_p\).

Fix an even integer \(k\). For \(1\leq u,v\leq N_k\) and \(1\leq t,t'\leq n_k+1,\) we have
\begin{align} \label{eq-28-08-1442}
\widetilde a_k(u)+t-1 \not= \widetilde a_k(v)+t'-1\quad \mbox{if } \, (u,t)\not=(v,t').
\end{align}
Indeed, if \(u<v\) and \(1 \leq t',t \leq n_k+1\), then
\begin{align*}
&\bigl(\widetilde a_k(v)+t'-1\bigr)
-\bigl(\widetilde a_k(u)+t-1\bigr)
\overset{\eqref{cond:5-even-spacing}}{=}(v-u)\Theta_{n_k+1}+t'-t
\geq
\Theta_{n_k+1}-n_k
\geq1.
\end{align*}
Moreover, \eqref{eq:5-local-cover} gives
\begin{align} \label{eq-28-08-1444}
\sigma_{m_k}<r_k<m_k+1
\qquad\text{and}\qquad
\|\mathbf w_{r_k}\|_\infty<m_k+2.
\end{align}
Since \(\|y_k\|_p<1\), we obtain
\begin{align*}
\sum_{u=1}^{N_k}
\left\|
S_{\mathbf w_{r_k}}^{\widetilde a_k(u)}y_k
\right\|_p^p &\quad=
\sum_{u=1}^{N_k}\sum_{t=1}^{n_k+1}
|z_{k,t-1}|^p
\frac{\displaystyle
 \prod_{\ell=1}^{t-1}
 \left(1+\frac{r_k}{\ell}\right)^p}
{\displaystyle
 \prod_{\ell=1}^{\widetilde a_k(u)+t-1}
 \left(1+\frac{r_k}{\ell}\right)^p}
\\
&\hspace{-0.1in}\overset{\eqref{eq-28-08-1442},\eqref{eq-28-08-1444}}{\leq}
\sum_{d\geq\widetilde a_k(1)}
\frac{(m_k+2)^{(n_k+1)p}}
{\displaystyle
 \prod_{\ell=1}^{d}
 \left(1+\frac{\sigma_{m_k}}{\ell}\right)^p}
\overset{\eqref{cond:5-even-tail}}{<}2^{-k}.
\end{align*}
Hence, the preceding estimate and the pairwise disjointness of the
supports show that the series in \eqref{eq-sec2-1-open} converges in
\(\ell_p\), and
\begin{align}
\|\widetilde x_q\|_p^p
=
\sum_{\substack{k\in2\mathbb N\\q_k=q}}
\sum_{u=1}^{N_k}
\left\|
S_{\mathbf w_{r_k}}^{\widetilde a_k(u)}y_k
\right\|_p^p<
\sum_{k\geq1}2^{-k}
=1.
\label{eq:5-norm-xtilde}
\end{align} 

The coordinate \(\widetilde b_{2q+1}(1)\) belongs to the odd interval
\(\widetilde I_{2q+1}(1)\), whereas \(\widetilde x_q\) is supported
on even intervals. Hence
\begin{align}
\pi_{\widetilde b_{2q+1}(1)}(\widetilde x_q)=0,
\qquad
\pi_{\widetilde b_{2q+1}(1)}(x_q)=1.
\label{eq:5-distinguished-coordinate}
\end{align}
It follows from \eqref{eq:5-norm-xtilde} and \eqref{eq:5-distinguished-coordinate} that
\begin{align} \label{eq-28-08-1514}
1
\leq
\|x_q\|_p^p
=
1+\varepsilon^p\|\widetilde x_q\|_p^p
\leq
1+\varepsilon^p
\leq
(1+\varepsilon)^p.
\end{align}

Observe that
\[
\operatorname{supp}(x_q)
\subseteq
\left\{\widetilde b_{2q+1}(1)\right\}
\cup
\bigcup_{\substack{k\in2\mathbb N\\q_k=q}}
\bigcup_{u=1}^{N_k}\widetilde I_k(u).
\]
Each even integer \(k\) has exactly one associated value \(q_k\), and
the coordinate intervals constructed in
Subsection~\ref{Subsec-5.1} are pairwise disjoint. Therefore
\begin{align}
\operatorname{supp}(x_q)\cap\operatorname{supp}(x_{q'})
=\emptyset
\qquad(q\neq q').
\label{eq:5-disjoint-xq}
\end{align} 

Define
\begin{align}
T:\ell_p&\longrightarrow\ell_p,
\nonumber\\
T\bigl((\lambda_q)_{q\geq0}\bigr)
&:=
\sum_{q\geq0}\lambda_qx_q.
\label{eq:5-definition-T}
\end{align}
The series converges in \(\ell_p\) by \eqref{eq-28-08-1514} and \eqref{eq:5-disjoint-xq}. Moreover, by using \eqref{eq:5-distinguished-coordinate}, \eqref{eq:5-disjoint-xq} and \eqref{eq-28-08-1514}, we have
\[
\|\lambda\|_p
\leq \|T(\lambda)\|_p
=
\bigg(\sum_{q\geq0}|\lambda_q|^p\|x_q\|_p^p\bigg)^{1/p} \leq
(1+\varepsilon)\|\lambda\|_p, \qquad \lambda \in \ell_p.
\]
Thus \(T\) is a \((1+\varepsilon)\)-isomorphic embedding, and its
range is closed and infinite-dimensional.

\subsection{The range of $T$ lies, apart from the zero vector, in
$\displaystyle\bigcap_{\mu>1}\mathrm{UFHC}(B_{\mathbf w_\mu})$}
\label{Subsec-5.3}

We prove that
\[
T(\ell_p)\setminus\{0\}
\subseteq
\bigcap_{\mu>1}\mathrm{UFHC}(B_{\mathbf w_\mu}).
\]

Fix \(x=T(\lambda)\neq0\), where
\(\lambda=(\lambda_q)_{q\geq0}\in\ell_p\), and put
\[
q_0:=\min\{q\in\mathbb N_0:\lambda_q\neq0\}.
\]
Since multiplication by a nonzero scalar preserves upper frequent
hypercyclicity, we may assume that
\begin{align}
x=\sum_{q\geq q_0}\lambda_qx_q,
\qquad
\lambda_{q_0}=1.
\label{eq:5-expansion-x}
\end{align}
Fix \(\mu>1\). We shall prove that
\(x\in\mathrm{UFHC}(B_{\mathbf w_\mu})\).

As in Subsection~\ref{subsection-3.3}, we apply the upper-frequent
Bourdon--Feldman criterion to the nonempty open set
\(\varepsilon B_{\ell_p}\). Let \(W\) be a nonempty open subset of
\(\varepsilon B_{\ell_p}\), and choose \(z_0\in W\) and \(\delta>0\)
such that \(B(z_0,\delta)\subset W.\) We set
\begin{align}
\eta_0
:=
\min\left\{
\frac{\delta^p}{8^p\varepsilon^p},
\frac{\delta^p}
{\varepsilon^p2^{p+1}\|\lambda\|_p^p}
\right\}.
\label{eq:5-eta0}
\end{align}

We may choose \(n_0\in\mathbb N_0\) and
\(y_0\in\mathscr D_{n_0}^{\mathbb K,p}\) such that
\begin{align} \label{eq:5-target-approximation}
\|\varepsilon y_0-z_0\|_p
<\frac{\delta}{4}, \quad
\sigma_{n_0+1} 
\leq\mu, \quad
2^{-(n_0+1)}
<\frac{\eta_0}{2}.
\end{align}
Indeed, first choose a rational finitely supported vector \(z\) in
\(\varepsilon^{-1}B(z_0,\delta/4)\). We may then take \(n_0\) larger
than every index in \(\operatorname{supp}(z)\) and sufficiently large
that the last two conditions in  \eqref{eq:5-target-approximation} hold. For a
sufficiently small nonzero \(\beta\in\mathbb Q_{\mathbb K}\), the
vector \(y_0:=z+\beta e_{n_0}\) belongs to
\(\mathscr D_{n_0}^{\mathbb K,p}\) and satisfies
\eqref{eq:5-target-approximation}. Choose also \(h_0\in\mathbb N\) and
\(m_0\geq2\) such that
\begin{align}
\varepsilon2^{-h_0}<\frac{\delta}{8},
\qquad
\mu\in [\sigma_{m_0},m_0].
\label{eq:5-h-m-choice}
\end{align}

The tuple \((q_0,n_0,y_0,m_0,h_0)\) occurs infinitely often in
\((\mathcal R_s)_{s\geq1}\). Write 
\[
\{s\geq1:\mathcal R_s=(q_0,n_0,y_0,m_0,h_0)\}
=
\{s_1<s_2<\cdots\}.
\] 
For every \(i \in \mathbb{N}\), it follows from \eqref{cond:5-parameter-cover} that the finite family
\((J_k)_{k\in\mathcal B_{s_i}}\) covers \([\sigma_{m_0},m_0]\). Hence, for each $i \in \mathbb N,$ we may 
choose \(k_i\in\mathcal B_{s_i}\) such that
\begin{align}
\mu\in J_{k_i}.
\label{eq:5-selected-parameter}
\end{align} 
By \eqref{cond:5-parameter-cover}, for any $i \in \mathbb{N},$
\begin{align}
(q_{k_i},n_{k_i},y_{k_i},m_{k_i},h_{k_i})
=
(q_0,n_0,y_0,m_0,h_0).
\label{eq:5-selected-data}
\end{align}
Moreover, since the sets \((\mathcal B_s)_{s\geq1}\) are successive
sets of even integers, the sequence \((k_i)_{i\geq1}\) is
strictly increasing and hence tends to infinity.  

Choose \(Q\in\mathbb N\) such that
\begin{align}
Q >q_0, \quad
2Q+1&\geq\mu, \quad
\sum_{q\geq Q}|\lambda_q|^p
<\frac{\delta^p}{2^{p+2}}.
\label{eq:5-choice-Q}
\end{align} 
Since \(k_i\to\infty\), it follows from
\eqref{eq:5-target-approximation} that there exists \(i_0\in\mathbb N\)
such that, for every \(i\geq i_0\),
\begin{align} \label{eq:5-selected-large-k}
\widetilde a_{k_i}(1)>
\max_{0\leq q<Q}\widetilde b_{2q+1}(1), \quad 
k_i-1
\geq\frac{\mu}{\log 2}, \quad 
2^{-(n_0+1)}+\sum_{\ell>k_i}2^{-\ell}
<\eta_0.
\end{align} 

Define
\begin{align}
\mathcal S
:=
\bigcup_{i \geq i_0}
\left\{
\widetilde a_{k_i}(u):
1\leq u\leq N_{k_i}
\right\}.
\label{eq:5-return-set}
\end{align}
Fix \(i \geq i_0\) and \(1\leq u'\leq N_{k_i}\), and set
\begin{align}
s':=\widetilde a_{k_i}(u').
\label{eq:5-fixed-return-time}
\end{align}
For this return time, put
\begin{align} \label{eq:5-return-errors}
\begin{split}
E_1^{(5)}(s')
&:=
\left\|
\varepsilon B_{\mathbf w_\mu}^{s'}\widetilde x_{q_0}-z_0
\right\|_p,\\
E_2^{(5)}(s')
&:=
\sum_{q\geq q_0}|\lambda_q|^p
\left\|
B_{\mathbf w_\mu}^{s'}
e_{\widetilde b_{2q+1}(1)}
\right\|_p^p,\\
E_3^{(5)}(s')
&:=
\varepsilon^p
\sum_{q>q_0}|\lambda_q|^p
\left\|
B_{\mathbf w_\mu}^{s'}\widetilde x_q
\right\|_p^p.
\end{split}
\end{align}

By \eqref{eq:5-expansion-x} and the definition of \(x_q\),
\begin{align*} 
B_{\mathbf w_\mu}^{s'}x-z_0
=
\left(
\varepsilon B_{\mathbf w_\mu}^{s'}\widetilde x_{q_0}-z_0
\right)
+
\sum_{q\geq q_0}\lambda_q
B_{\mathbf w_\mu}^{s'}e_{\widetilde b_{2q+1}(1)}
+
\varepsilon\sum_{q>q_0}\lambda_q
B_{\mathbf w_\mu}^{s'}\widetilde x_q.
\end{align*}
The supports of the vectors occurring in the last two sums are
pairwise disjoint. Indeed, their supports are obtained by applying the
same backward shift to pairwise disjoint coordinate intervals.
Thus,
\begin{align}
\left\|B_{\mathbf w_\mu}^{s'}x-z_0\right\|_p
\leq
E_1^{(5)}(s')+\bigl(E_2^{(5)}(s')+E_3^{(5)}(s')\bigr)^{1/p}.
\label{eq:5-error-decomposition}
\end{align}
{\begin{claim}
\label{claim:5-return-time-estimates}
For every \(s'\) chosen in \eqref{eq:5-fixed-return-time},
\[
E_1^{(5)}(s')<\frac{\delta}{2},
\qquad
E_2^{(5)}(s')<\frac{\delta^p}{2^{p+1}},
\qquad
E_3^{(5)}(s')<\frac{\delta^p}{2^{p+1}}.
\]
\end{claim}}

The proof of Claim~\ref{claim:5-return-time-estimates} is given in
Subsection~\ref{Subsec-7.3}. The claim and
\eqref{eq:5-error-decomposition} give
\[
\left\|B_{\mathbf w_\mu}^{s'}x-z_0\right\|_p
<
\frac{\delta}{2}
+
\left(
\frac{\delta^p}{2^{p+1}}
+
\frac{\delta^p}{2^{p+1}}
\right)^{1/p}
=
\delta.
\]
Hence
\begin{align}
\mathcal S
\subseteq
\mathcal N_{B_{\mathbf w_\mu}}
\bigl(x,B(z_0,\delta)\bigr).
\label{eq:5-return-set-inclusion}
\end{align}

It remains to estimate the upper density of \(\mathcal S\). By
\eqref{eq:5-selected-data}, \(n_{k_i}=n_0\) for every \(i\). Thus, for every \(i\geq i_0\), we have
\begin{align*}
\frac{
\left|
\mathcal S\cap
\llbracket
0,\widetilde a_{k_i}(N_{k_i})
\rrbracket
\right|
}{
\widetilde a_{k_i}(N_{k_i})+1
}
 \geq
\frac{N_{k_i}}
{\widetilde a_{k_i}(N_{k_i})+1}
\overset{\eqref{cond:5-even-density}}{\geq}
\frac1{2\Theta_{n_0+1}}.
\end{align*}
Since \(\widetilde a_{k_i}(N_{k_i})\to\infty\) as \(i \to \infty\), it follows that
\begin{align}
\overline{\mathrm d}(\mathcal S)
\geq
\frac1{2\Theta_{n_0+1}}
>0.
\label{eq:5-positive-density-S}
\end{align}
Combining \eqref{eq:5-return-set-inclusion} and
\eqref{eq:5-positive-density-S}, we conclude that
\[
\overline{\mathrm d}\left(
\mathcal N_{B_{\mathbf w_\mu}}
\bigl(x,B(z_0,\delta)\bigr)
\right)>0.
\]
Since \(B(z_0,\delta)\subset W\), the return set of \(x\) to \(W\)
also has positive upper density. As \(W\) was an arbitrary nonempty
open subset of \(\varepsilon B_{\ell_p}\), the upper-frequent
Bourdon--Feldman criterion yields
\(x\in\mathrm{UFHC}(B_{\mathbf w_\mu})\). Finally, \(\mu>1\) was
arbitrary, and therefore
\[
T(\ell_p)\setminus\{0\}
\subseteq
\bigcap_{\mu>1}\mathrm{UFHC}(B_{\mathbf w_\mu}).
\]

\subsection{The range of $T$ does not intersect
$\displaystyle\bigcup_{\mu>1}\mathrm{FHC}(B_{\mathbf w_\mu})$}
\label{Subsec-5.4}

Set
\[
\mathcal M
:=
\bigcup_{k\geq1}\bigcup_{u=1}^{N_k}
\widetilde I_k(u).
\]
By the definition of the vectors \(x_q\) and the support inclusion
\eqref{eq-10-08-A},
\begin{align}
\operatorname{supp}(x)\subseteq\mathcal M
\qquad\bigl(x\in T(\ell_p)\bigr).
\label{eq:5-support-range}
\end{align}

We first observe that \(\mathcal M\) has lower density zero. Since the
coordinate intervals occur in their natural order, for every \(k\geq2\),
\[
\llbracket
\widetilde b_{k-1}(N_{k-1})+1,
\widetilde a_k(1)-1
\rrbracket
\subseteq
\mathbb N_0\setminus\mathcal M.
\] 
Since \(0\notin\mathcal M\), the preceding inclusion gives
\begin{align*}
\frac{
\left|
(\mathbb N_0\setminus\mathcal M)\cap
\llbracket0,\widetilde a_k(1)-1\rrbracket
\right|
}{
\widetilde a_k(1)
}
\geq
\frac{
\widetilde a_k(1)
-\widetilde b_{k-1}(N_{k-1})
}{
\widetilde a_k(1)
}
\overset{\eqref{cond:5-gap}}{\geq}
\frac{k-1}{k} \xrightarrow{k\to\infty} 1.
\end{align*}
We therefore obtain
\begin{align} \label{eq-03-09-2026-1727}
\overline{\mathrm d}(\mathbb N_0\setminus\mathcal M)=1.
\end{align}

Now fix \(x\in T(\ell_p)\) and \(\mu>1\). If
\(j\in\mathbb N_0\setminus\mathcal M\), then
\(\pi_j(x)=0\) by \eqref{eq:5-support-range}, and therefore
\begin{align} \label{eq-13-09-2026-1729}
\pi_0\bigl(B_{\mathbf w_\mu}^{j}x\bigr)
=
\left[
\prod_{i=1}^{j}
\left(1+\frac{\mu}{i}\right)
\right]\pi_j(x)
=0.
\end{align} 
From \eqref{eq-03-09-2026-1727},  \eqref{eq-13-09-2026-1729}, and Lemma~\ref{lem-1-subs-upper}\ref{item-b-15-06-1929}, it follows that 
\(x\notin\mathrm{FHC}(B_{\mathbf w_\mu})\). This completes the proof.

\section{Technical Lemmas}
\label{Sc-6} 

This section contains the proofs of the technical lemmas used in
Sections~3, 4 and 5.

\begin{lemma}\label{sublemma-1}
Let $n\in\mathbb N$, and let
$\delta_1,\ldots,\delta_{2^n}\in\mathbb R$ satisfy
$\delta_j\ge-1$ for every $j$. If $n\ge7$ and
\begin{equation}\label{eq-02-08-D}
\sum_{j=1}^{2^n}\delta_j\le1,
\end{equation}
then there exist integers $a$ and $b$ such that
\begin{equation*}
I_n:=\llbracket a,b\rrbracket
\subseteq\llbracket1,2^n\rrbracket,
\qquad
b-a+1=n,
\qquad
\max_{a\le k\le b}\sum_{j=k}^{b}\delta_j\le1.
\end{equation*}
\begin{proof}
Set $s_0:=2^n+1$ and suppose, for contradiction, that the
conclusion fails. Recursively, as long as $s_k\ge2n$, apply this
failure to the length-$n$ interval
$\llbracket s_k-n,s_k-1\rrbracket\subseteq
\llbracket1,2^n\rrbracket$ to choose
$s_{k+1}\in\llbracket s_k-n,s_k-1\rrbracket$ such that
\begin{equation}\label{eq-02-08-B}
\sum_{j=s_{k+1}}^{s_k-1}\delta_j>1.
\end{equation}
The sequence $(s_k)$ is strictly decreasing, so the construction
stops at an index $m$ for which
\begin{equation}\label{eq-23-08-2051}
n\le s_m<2n.
\end{equation}
Moreover, each of the $m$ consecutive intervals
$\llbracket s_{k+1},s_k-1\rrbracket$, $0\le k<m$, has at most
$n$ elements. Hence
\[
2^n+1-s_m=s_0-s_m\le mn,
\]
and therefore
\begin{equation}\label{eq-23-08-2123}
m\ge\frac{2^n+1-s_m}{n}>\frac{2^n}{n}-2.
\end{equation}
Using $\delta_j\ge-1$, \eqref{eq-02-08-B},
\eqref{eq-23-08-2051}, and \eqref{eq-23-08-2123}, we obtain
\begin{align*}
\sum_{j=1}^{2^n}\delta_j
=
\sum_{j=1}^{s_m-1}\delta_j
+\sum_{k=0}^{m-1}\sum_{j=s_{k+1}}^{s_k-1}\delta_j
>-(s_m-1)+m
>-2n+1+\frac{2^n}{n}-2
>1,
\end{align*}
where the last inequality follows from
$2^n>2n(n+1)$ for $n\ge7$. This contradicts
\eqref{eq-02-08-D}.
\end{proof}
\end{lemma}

\begin{lemma}\label{lemma-pesos}
Suppose that ${\bf w}=(\omega_k)_{k\ge1}$
is a weight sequence that satisfies the second
condition in \eqref{eq-19-08-I}. Then there exists $\beta\ge1$
for which the set $\mathcal A^{\mathbf w}_{N,\beta}$ of
\eqref{eq-06-08-B} is infinite for every $N\in\mathbb N$.
\begin{proof}
Since both the second condition in \eqref{eq-19-08-I} and the
definition of $\mathcal A^{\mathbf w}_{N,\beta}$ depend only on
the moduli of the weights, we may replace $\mathbf w$ by
$|\mathbf w|=(|\omega_k|)_{k\ge1}$. Thus, throughout the proof,
we assume that $\omega_k>0$ for every $k$.

First suppose that \(\inf_{k\ge1}\omega_k=0.\)
We show that $\beta=1$ works. Fix $N\in\mathbb N$ and set
\[
M_N:=\max_{0\le r\le N-1}\|\mathbf w\|_\infty^r.
\]
Given $C\in\mathbb N$, the equality
$\inf_{k\ge1}\omega_k=0$, together with the fact that all weights
are nonzero, allows us to choose
$K\ge C+N$ such that
\begin{equation}\label{eq-24-08-1437}
\omega_K\le M_N^{-1}.
\end{equation}
Set $q:=K-N$. For every $u$ with $q<u\le q+N=K$, we have
\[
\prod_{j=u}^{q+N}\omega_j
=\omega_K\prod_{j=u}^{K-1}\omega_j
\le\omega_K M_N\le1.
\]
Thus $q\in\mathcal A^{\mathbf w}_{N,1}$. Since $q\ge C$ and
$C$ was arbitrary, $\mathcal A^{\mathbf w}_{N,1}$ is unbounded,
and hence infinite.

Now suppose that
\(
a:=\inf_{k\ge1}\omega_k>0,
\)
and define
\begin{equation}\label{eq-03-08-F}
\beta:=\max\big\{2,a^{-1}\big\}>1.
\end{equation}
It suffices to prove the result for $N\ge7$. Indeed, if
$1\le N<7$ and $k\in\mathcal A^{\mathbf w}_{7,\beta}$, then
$k+7-N\in\mathcal A^{\mathbf w}_{N,\beta}$.

Fix $N\ge7$ and set
\[
\mathcal C_{N,\beta}
:=
\left\{k\in\mathbb N:
\prod_{j=k+1}^{k+2^N}\omega_j\le\beta\right\}.
\]
We first prove that $\mathcal C_{N,\beta}$ is infinite. Otherwise,
there would exist $k_0\in\mathbb N$ such that
\begin{equation}\label{eq-24-08-1525}
\prod_{j=k+1}^{k+2^N}\omega_j>\beta
\qquad\text{for every }k\ge k_0.
\end{equation}
Choose $m\in\mathbb N$ so large that
\begin{equation}\label{eq-03-08-C}
\beta^{m-1}>a^{-k_0},
\end{equation}
and put $L:=k_0+m2^N$. For every $r\in\mathbb N_0$, split the
product of $L$ consecutive weights after $r$ into an initial block
of length $k_0$ and $m$ blocks of length $2^N$. By
\eqref{eq-24-08-1525},
\begin{align*}
\prod_{j=r+1}^{r+L}\omega_j
=
\left(\prod_{j=r+1}^{r+k_0}\omega_j\right)
\prod_{\ell=0}^{m-1}
\left(
\prod_{j=r+k_0+\ell2^N+1}^{r+k_0+(\ell+1)2^N}\omega_j
\right) >a^{k_0}\beta^m>\beta>1,
\end{align*}
where the penultimate inequality follows from \eqref{eq-03-08-C}.
Hence,
\[
\sup_{n\ge1}\inf_{r\ge0}\prod_{j=1}^{n}\omega_{r+j}
\ge
\inf_{r\ge0}\prod_{j=1}^{L}\omega_{r+j}
\ge a^{k_0}\beta^m>1,
\]
contradicting the second condition in \eqref{eq-19-08-I}.
Therefore $\mathcal C_{N,\beta}$ is infinite.

Fix $k'\in\mathcal C_{N,\beta}$ and define
\[
\delta_\ell:=\log_\beta(\omega_{k'+\ell}),
\qquad 1\le\ell\le2^N.
\]
By the definition of $\mathcal C_{N,\beta}$,
\begin{equation*}
\sum_{\ell=1}^{2^N}\delta_\ell
=
\log_\beta\left(
\prod_{j=k'+1}^{k'+2^N}\omega_j
\right)
\le1.
\end{equation*}
Moreover, \eqref{eq-03-08-F} gives
\begin{equation*}
-1=\log_\beta(\beta^{-1})
\le\log_\beta(a)
\le\delta_\ell,
\qquad 1\le\ell\le2^N.
\end{equation*}
By Lemma~\ref{sublemma-1}, there exist
$1\le r\le s\le2^N$ such that
\begin{equation}\label{eq-2025-10-31}
s-r+1=N,
\qquad
\max_{r\le t\le s}\sum_{\ell=t}^{s}\delta_\ell\le1.
\end{equation}
Set $q:=k'+r-1$. Then $q+N=k'+s$, and, for every
$u$ with $q<u\le q+N$, \eqref{eq-2025-10-31} yields
\begin{equation}\label{eq-2025-10-31-2154}
\prod_{j=u}^{q+N}\omega_j
=
\beta^{\sum_{\ell=u-k'}^{s}\delta_\ell}
\le\beta.
\end{equation}
Hence \(q\in\mathcal A^{\mathbf w}_{N,\beta}
\cap\llbracket k',k'+2^N\rrbracket.\) Since $\mathcal C_{N,\beta}$ is unbounded and $q\ge k'$, the set $\mathcal A^{\mathbf w}_{N,\beta}$ is unbounded, and therefore infinite.
\end{proof}
\end{lemma}

\begin{lemma}
\label{lem:exponential-tail}
Let \(\lambda>1/p\) and \(M>0\). There exists an integer \(C\geq2\)
such that, for each $k\geq 2,$
\begin{align}
M\sum_{d\geq C^k}
\frac{1}{
\displaystyle
\prod_{\ell=1}^{d}
\left(1+\frac{\lambda}{\ell}\right)^p}
<2^{-k}.
\label{eq:exponential-tail}
\end{align}
\end{lemma}
\begin{proof}
Since 
\begin{align*}
\frac{d^\lambda}
{\displaystyle
\prod_{\ell=1}^{d}
\left(1+\frac{\lambda}{\ell}\right)}= \frac{d!\,d^\lambda}
{\displaystyle
 d!
\prod_{\ell=1}^{d}
\left(1+\frac{\lambda}{\ell}\right)} =
\frac{d!\,d^\lambda}
{(\lambda+1)(\lambda+2)\cdots(\lambda+d)} \xrightarrow{d\to\infty} \Gamma(\lambda+1),
\end{align*}
where the convergence follows from \cite[Eq. (95), p.~193]{Rudin}, there exists \(A>0\) such that, for any $D \geq 2,$ we have
\[
M\sum_{d\geq D}
\frac{1}{
\displaystyle
\prod_{\ell=1}^{d}
\left(1+\frac{\lambda}{\ell}\right)^p} \leq A \sum_{d\geq D}
\frac{1}{ d^{\lambda p}}
\leq A \int_{D-1}^{\infty} \dfrac{dx}{x^{\lambda p}} = \dfrac{A(D-1)^{1-\lambda p}}{\lambda p - 1}.
\]
Set \(\beta:=\lambda p-1>0\), and choose an integer \(C\geq2\)
such that
\[
C^\beta>\max\left\{2,\frac{4A}{\beta}\right\}.
\]
Since \(C^k-1\geq C^{k-1}\), for every \(k\geq2\) we have
\begin{align*}
\frac{A(C^k-1)^{-\beta}}{\beta}
\leq
\frac{A}{\beta}C^{-(k-1)\beta} =
\frac{A}{\beta}C^{-\beta}
\left(C^{-\beta}\right)^{k-2} <
\frac14\left(\frac12\right)^{k-2}
=
2^{-k},
\end{align*}
which proves \eqref{eq:exponential-tail}.

\end{proof} 

\begin{lemma}
\label{lem:finite-completion}
For every \(s\geq1\), once the recursive construction reaches the stage \(s\), it passes from \(s\) to \(s+1\) after finitely many
steps.
\end{lemma}

\begin{proof}
Fix \(s\), write \(\mathcal R_s=(q,n,y,m,h)\), and suppose that the
construction never passes to \(s+1\). Let \(k_0\) be the first even
integer used for \(s\). Thus, the even integers used while \(s\) is
fixed are
\[
k_0,k_0+2,k_0+4,\ldots.
\]
Put
\[
H_d:=\sum_{\ell=1}^{d}\frac1\ell
\qquad(d\geq1).
\]

\noindent\emph{Step 1: Growth of the endpoints.}
The definition of the odd interval gives
\[
\widetilde b_{k+1}(1)
\leq
\bigl(2\lceil p\rceil(k+1)+1\bigr)
\widetilde b_k(N_k).
\]
Moreover, since \(N_{k+2}=\widetilde a_{k+2}(1)\) and
\(\Theta_{n+1}\geq n+1\),
\[
\widetilde b_{k+2}(N_{k+2})
\leq
\bigl(\Theta_{n+1}+1\bigr)\widetilde a_{k+2}(1).
\]
Using the definition of \(\widetilde a_{k+2}(1)\), together with
\[
C_{m,n+1}^{\,k+2}
\leq
C_{m,n+1}^{\,2}\widetilde b_k(N_k)
\qquad\text{and}\qquad
n+1\leq\widetilde b_k(N_k),
\]
we obtain
\[
\begin{aligned}
\widetilde b_{k+2}(N_{k+2})
&\leq
\bigl(\Theta_{n+1}+1\bigr)
\Bigl[
(k+2)\bigl(2\lceil p\rceil(k+1)+1\bigr)
+C_{m,n+1}^{\,2}+1
\Bigr]
\widetilde b_k(N_k)
\\
&\leq
(k+2)^3\widetilde b_k(N_k)
\end{aligned}
\]
for all sufficiently large even \(k\).

\noindent\emph{Step 2: A lower bound for \(\Delta_k\).}
Choose an even integer \(k_1\geq k_0\) such that the estimate obtained
in Step~1 holds for every even \(k\geq k_1\). By the monotonicity of
the harmonic numbers and the inequality \(H_{ab}\leq H_a+H_b\), it
follows that
\[
H_{\widetilde b_{k+2}(N_{k+2})}
\leq
H_{\widetilde b_k(N_k)}+3H_{k+2}
\qquad(k\geq k_1,\ k\in2\mathbb N).
\]
Iterating this estimate, for \(k=k_1+2j\), we obtain
\begin{align*}
H_{\widetilde b_k(N_k)}
&\leq
H_{\widetilde b_{k_1}(N_{k_1})}
+
3\sum_{i=1}^{j}H_{k_1+2i}
\\
&\leq
H_{\widetilde b_{k_1}(N_{k_1})}
+
3jH_k.
\end{align*}
Since \(j\leq k/2\) and \(kH_k\geq1\), it follows that
\[
H_{\widetilde b_k(N_k)}
\leq
\left(
H_{\widetilde b_{k_1}(N_{k_1})}+\frac32
\right)kH_k.
\]
Consequently,
\[
\Delta_k
\geq
\frac{2^{-h}}
{H_{\widetilde b_{k_1}(N_{k_1})}+\frac32}
\frac1{kH_k}
\qquad(k\geq k_1,\ k\in2\mathbb N).
\]

\noindent\emph{Step 3: Divergence of the comparison series.}
We claim that
\[
\sum_{k\in2\mathbb N}\frac1{kH_k}=\infty.
\]
Indeed, for every \(r\geq1\), there are exactly \(2^{r-1}\) even
integers \(k\) satisfying
\[
2^r\leq k<2^{r+1}.
\]
For each such \(k\),
\[
kH_k\leq2^{r+1}H_{2^{r+1}}.
\]
Moreover, repeated use of \(H_{ab}\leq H_a+H_b\) gives
\[
H_{2^{r+1}}\leq(r+1)H_2.
\]
Therefore,
\begin{align*}
\sum_{\substack{2^r\leq k<2^{r+1}\\ k\in2\mathbb N}}
\frac1{kH_k}
&\geq
\frac{2^{r-1}}{2^{r+1}H_{2^{r+1}}}
\geq
\frac1{4(r+1)H_2}.
\end{align*}
Since \(\sum_{r\geq1}1/(r+1)\) diverges, so does
\(\sum_{k\in2\mathbb N}1/(kH_k)\). Hence, by Step~2 and after omitting
finitely many initial terms,
\[
\sum_{j=0}^{\infty}\Delta_{k_0+2j}=\infty.
\]

\noindent\emph{Step 4: The contradiction.}
By our assumption, no step is terminal. Therefore,
\[
\eta_{k_0+2j}=\frac{\Delta_{k_0+2j}}4
\qquad(j\geq0).
\]
After \(N+1\) such steps, we consequently have
\[
\sigma_m+\frac14\sum_{j=0}^{N}\Delta_{k_0+2j}<m
\qquad(N\geq0).
\]
Letting \(N\to\infty\) contradicts Step~3.
\end{proof}

\section{Proofs of the Claims}
\label{section-technical-estimates}

This section contains the proofs of the technical estimates used in Subsections~\ref{subsection-3.3}, \ref{Subsec-4.3}, and~\ref{Subsec-5.3}.

\subsection{Proof of Claim~\ref{claim-return-time-estimates}}
\label{subsection-proof-claim-return}

Fix \(s'\in\mathcal S\) with \(s'>s_0\), and choose \(k'\geq1\) and
\(1\leq u_*\leq j_{r_{k'}}\) such that
\[
s'=a_{r_{k'}}(u_*).
\]
Put \(r_*:=r_{k'}\) and call \(I_{r_*}(u_*)\) the target interval.
Then \(r_*\in A_{n^*,q_0}\) and, by \eqref{29-07-E},
\[
\Phi_{n^*,q_0}^p(r_*)=y_{k'}.
\]
Recall that \(\|\mathbf w\|_\infty\geq1\) and
\[
\sum_{\ell\geq\theta(d)}
\frac{\|\mathbf w\|_\infty^{dp}}
{\bigl|\omega_1\cdots\omega_\ell\bigr|^p}
<2^{-d}
\qquad(d\geq0).
\]

\noindent\emph{Step 1: A common tail estimate.}
For \(q\in\mathbb N_0\), put
\[
\mathcal R_q(s')
:=
\sum_{n\geq0}
\sum_{r\in A_{n,q}}
\sum_{\substack{1\leq u\leq j_r\\
(r,u)\neq(r_*,u_*)}}
\left\|
B_{\mathbf w}^{s'}
S_{\mathbf w}^{a_r(u)}
\bigl(\Phi_{n,q}^p(r)\bigr)
\right\|_p^p.
\]
When \(q\neq q_0\), the last restriction is automatic.

Since the coordinate intervals are pairwise disjoint and occur in
their natural order, every interval distinct from the target interval
satisfies exactly one of the following:
\[
b_r(u)<s'
\qquad\text{or}\qquad
a_r(u)>b_{r_*}(u_*).
\]
In the first case,
\[
B_{\mathbf w}^{s'}
S_{\mathbf w}^{a_r(u)}
\bigl(\Phi_{n,q}^p(r)\bigr)=0.
\]

Consider the second case. Then \(a_r(u)>s'\), and hence
\[
B_{\mathbf w}^{s'}S_{\mathbf w}^{a_r(u)}
=
S_{\mathbf w}^{a_r(u)-s'}.
\]
Moreover,
\[
\Phi_{n,q}^p(r)
=
P_n\bigl(\Phi_{n,q}^p(r)\bigr)
\qquad\text{and}\qquad
\bigl\|\Phi_{n,q}^p(r)\bigr\|_p<1.
\]
Therefore,
\begin{align}
\begin{split} \label{eq-05-09-26-1413}
\left\|
B_{\mathbf w}^{s'}
S_{\mathbf w}^{a_r(u)}
\bigl(\Phi_{n,q}^p(r)\bigr)
\right\|_p^p
&=
\left\|
S_{\mathbf w}^{a_r(u)-s'}
P_n\bigl(\Phi_{n,q}^p(r)\bigr)
\right\|_p^p
\leq
\left\|
S_{\mathbf w}^{a_r(u)-s'}\circ P_n
\right\|^p
\\
&\hspace{-0.08in}\overset{\eqref{eq-20-08-0850}}{\leq}
\sum_{t=1}^{n+1}
\frac{\|\mathbf w\|_\infty^{np}}
{\bigl|
\omega_1\cdots
\omega_{a_r(u)-s'+t-1}
\bigr|^p}
\\&=
\sum_{\ell=a_r(u)-s'}^{a_r(u)-s'+n}
\frac{\|\mathbf w\|_\infty^{np}}
{\bigl|\omega_1\cdots\omega_\ell\bigr|^p}
=
\sum_{\ell=a_r(u)-s'}^{b_r(u)-s'}
\frac{\|\mathbf w\|_\infty^{np}}
{\bigl|\omega_1\cdots\omega_\ell\bigr|^p},
\end{split}
\end{align}
where the last equality follows from
\(b_r(u)-a_r(u)=n\).

For each block lying after the target interval, the set of indices in
the preceding estimate is the translated interval
\[
I_r(u)-s'
=
\llbracket a_r(u)-s',b_r(u)-s'\rrbracket.
\]
These intervals are pairwise disjoint, since they are obtained by
translating the pairwise disjoint coordinate intervals \(I_r(u)\) by
\(-s'\).

We first apply this observation to the remaining blocks with index
\(r_*\), which can occur only when \(q=q_0\). These blocks correspond
to \(u_*<u\leq j_{r_*}\). Since \(s'=a_{r_*}(u_*)\) and \(\tau(r_*)=n^*\), condition
\eqref{cond:3c} gives
\[
a_{r_*}(u)-s'
=
(u-u_*)\theta(n^*)
\geq
\theta(n^*)
\qquad(u_*<u\leq j_{r_*}).
\]
Therefore, the corresponding translated intervals are contained in
\(\llbracket\theta(n^*),\infty\rrbracket\).
Hence,
\begin{align*}
\sum_{u=u_*+1}^{j_{r_*}}
\left\|
B_{\mathbf w}^{s'}
S_{\mathbf w}^{a_{r_*}(u)}
\bigl(\Phi_{n^*,q_0}^p(r_*)\bigr)
\right\|_p^p
&\overset{\eqref{eq-05-09-26-1413}}{\leq}
\sum_{\ell\geq\theta(n^*)}
\frac{\|\mathbf w\|_\infty^{n^*p}}
{\bigl|\omega_1\cdots\omega_\ell\bigr|^p}
<2^{-n^*}.
\end{align*}

We now estimate the contribution of the blocks with indices
\(r>r_*\). For
\(n,q\in\mathbb N_0\), define
\[
\rho_{n,q}
:=
\min\{r\in A_{n,q}:r>r_*\},
\qquad
c_{n,q}
:=
b_{\rho_{n,q}-1}(1).
\]
Since
\(\rho_{n,q}\) and \(r_*\) are odd, we have
\(\rho_{n,q}-1>r_*\). Thus the interval indexed by
\(\rho_{n,q}-1\) lies strictly to the right of the target interval,
and hence
\[
c_{n,q}\geq s'+1.
\]
Moreover, since \(\tau(\rho_{n,q})=n\), condition
\eqref{cond:3a} gives 
\begin{align} \label{eq-05-09-26-1415}
a_{\rho_{n,q}}(1)-s'
\geq
a_{\rho_{n,q}}(1)-c_{n,q}
\geq
\theta(n+c_{n,q}).
\end{align}
The disjointness of the translated intervals and the defining
property of \(\theta\) now give 
\begin{align*}
\sum_{\substack{r\in A_{n,q}\\r\geq\rho_{n,q}}}
\sum_{u=1}^{j_r}
\left\|
B_{\mathbf w}^{s'}
S_{\mathbf w}^{a_r(u)}
\bigl(\Phi_{n,q}^p(r)\bigr)
\right\|_p^p
&\overset{\eqref{eq-05-09-26-1413}}{\leq}
\sum_{\ell\geq a_{\rho_{n,q}}(1)-s'}
\frac{\|\mathbf w\|_\infty^{np}}
{\bigl|\omega_1\cdots\omega_\ell\bigr|^p}
\\
&\overset{\eqref{eq-05-09-26-1415}}{\leq}
\sum_{\ell\geq\theta(n+c_{n,q})}
\frac{\|\mathbf w\|_\infty^{(n+c_{n,q})p}}
{\bigl|\omega_1\cdots\omega_\ell\bigr|^p}
\\&\overset{\eqref{eq-11-06-1534}}{<}
2^{-(n+c_{n,q})}.
\end{align*}
Hence, for every \(q\in\mathbb N_0\),
\begin{align}
\begin{split}
\mathcal R_q(s')
&\leq
2^{-n^*}
+
\sum_{n\geq0}2^{-(n+c_{n,q})}
\\
&\leq
2^{-n^*}
+
\sum_{n\geq0}2^{-(n+s'+1)}
\\&=
2^{-n^*}+2^{-s'}
<\eta_0.
\label{eq-common-error-estimate}
\end{split}
\end{align}
The last inequality follows from \(n^*\geq n_0\),
\(s'\geq s_0+1\), \eqref{eq-25-07-E}, and
\eqref{eq-30-07-B}. 

\noindent\emph{Step 2: Estimate of \(E_1^{(3)}(s')\).}
The target block satisfies
\[
B_{\mathbf w}^{s'}
S_{\mathbf w}^{a_{r_*}(u_*)}
\bigl(\Phi_{n^*,q_0}^p(r_*)\bigr)
=
y_{k'}.
\]
Therefore, using the definition of \(\widetilde x_{q_0}\), the
disjointness of the remaining supports, \eqref{eq-25-07-E}, \eqref{29-07-C}, and
\eqref{eq-common-error-estimate}, we obtain
\begin{align*}
E_1^{(3)}(s')
&\leq
\|\varepsilon y_{k'}-y_0\|_p
+
\varepsilon\bigl(\mathcal R_{q_0}(s')\bigr)^{1/p}
\\
&<
\frac{\delta}{4}
+
\varepsilon\eta_0^{1/p}
\leq
\frac{\delta}{2}.
\end{align*}

\noindent\emph{Step 3: Estimate of \(E_2^{(3)}(s')\).}
Since \(s'>s_0\geq l_1\), by \eqref{eq-07-08-G},
\begin{align} \label{eq-05-09-26-1537}
B_{\mathbf w}^{s'}e_{b_{2q}(1)}
=
B_{\mathbf w}^{s'-l_1}
\left(
B_{\mathbf w}^{l_1}e_{b_{2q}(1)}
\right)
=0
\qquad(q<l_0).
\end{align}

Now fix \(q\geq l_0\) such that
\[
B_{\mathbf w}^{s'}e_{b_{2q}(1)}\neq0.
\]
Then \(s'\leq b_{2q}(1)\). Equality is impossible because \(s'\)
belongs to the target interval \(I_{r_*}(u_*)\), whereas
\(b_{2q}(1)\) belongs to the disjoint interval \(I_{2q}(1)\).
Therefore,
\begin{align} \label{eq-05-09-26-1535}
b_{2q}(1)>s' \qquad \text{whenever} \quad q\geq l_0 \, \, \text{and} \, \, B_{\mathbf w}^{s'}e_{b_{2q}(1)}\neq0.
\end{align}
It follows that \(I_{2q}(1)\) occurs after the target interval.
By the natural ordering of the coordinate intervals,
\[
s'\leq b_{2q-1}(j_{2q-1}).
\]
By \eqref{cond:2c},
\[
\begin{aligned}
a_{2q}(1)
&=
b_{2q}(1)-b_{2q-1}(j_{2q-1})+1
 \leq
b_{2q}(1)-s'+1
\leq
b_{2q}(1).
\end{aligned}
\] 
Hence, we may apply \eqref{eq-even-product-control} with
\[
k=2q
\qquad\text{and}\qquad
u=b_{2q}(1)-s'+1.
\]
This gives
\[
\prod_{\ell=b_{2q}(1)-s'+1}^{b_{2q}(1)}
|\omega_\ell|
\leq\beta.
\]
Consequently,
\begin{align*}
E_2^{(3)}(s')
&\overset{\eqref{eq-05-09-26-1537},\eqref{eq-05-09-26-1535}}{=}
\sum_{\substack{q\geq l_0\\ b_{2q}(1)>s'}}
|\lambda_q|^p
\prod_{\ell=b_{2q}(1)-s'+1}^{b_{2q}(1)}
|\omega_\ell|^p
\\
&\hspace{0.2in} \leq \quad
\beta^p\sum_{q\geq l_0}|\lambda_q|^p
\overset{\eqref{eq-05-09-26-1542}}{<}
\frac{\delta^p}{2^{p+1}}.
\end{align*}

\noindent\emph{Step 4: Estimate of \(E_3^{(3)}(s')\).}
If \(q>q_0\), the target pair \((r_*,u_*)\) does not occur in
\(\widetilde x_q\). Hence the disjointness of the translated supports
gives
\[
\left\|
B_{\mathbf w}^{s'}\widetilde x_q
\right\|_p^p
=
\mathcal R_q(s').
\]
It follows from \eqref{eq-common-error-estimate} that
\begin{align*}
E_3^{(3)}(s') <
\varepsilon^p\eta_0\|\lambda\|_p^p
\overset{\eqref{eq-25-07-E}}{\leq}
\frac{\delta^p}{2^{p+1}}.
\end{align*}
This proves Claim~\ref{claim-return-time-estimates}.

\subsection{Proof of
Claim~\ref{claim-sec4-return-time-estimates}}
\label{subsec:proof-claim-sec4-return-time-estimates}

\noindent\emph{Step 1: A common remainder estimate.}
For \(q\in\mathbb N_0\) and \(k\geq1\), set
\[
R_q(s_k)
:=
\sum_{n\geq0}
\sum_{\substack{r\in A_{n,q}\\r>r_k}}
B_{\mathbf w}^{s_k}S_{\mathbf w}^{c_r}
\bigl(\Phi_{n,q}^{p}(r)\bigr).
\]
If \(r<r_k\), the natural ordering of the intervals gives
\[
d_r<c_{r_k}=s_k.
\]
Since
\[
\operatorname{supp}\left(
S_{\mathbf w}^{c_r}
\bigl(\Phi_{n,q}^p(r)\bigr)
\right)
\subseteq
J_r
\subseteq
\llbracket0,s_k-1\rrbracket,
\]
it follows that
\[
B_{\mathbf w}^{s_k}S_{\mathbf w}^{c_r}
\bigl(\Phi_{n,q}^p(r)\bigr)=0.
\] 
Moreover, \(r_k\) belongs to \(A_{n^*,q_0}\) and to no other
\(A_{n,q}\), and
\[
B_{\mathbf w}^{s_k}S_{\mathbf w}^{c_{r_k}}
\bigl(\Phi_{n^*,q_0}^{p}(r_k)\bigr)
=
y_k.
\]
Thus,
\begin{align}
B_{\mathbf w}^{s_k}\widetilde x_q
=
\begin{cases}
y_k+R_{q_0}(s_k),&q=q_0,\\
R_q(s_k),&q\neq q_0.
\end{cases}
\label{eq:7-sec4-decomposition}
\end{align}

We now estimate \(R_q(s_k)\). Fix \(r\in A_{n,q}\) with
\(r>r_k\). Since \(r\) and \(r_k\) are odd, we have
\(r-1>r_k\), and the natural ordering of the intervals gives
\(d_{r-1}\geq s_k\). Thus, by \eqref{cond:4-odd-a} and the
monotonicity of \(\theta\),
\begin{align} \label{eq-05-09026-1656}
c_r-s_k
\geq
c_r-d_{r-1}
\geq
\theta(n+d_{r-1})
\geq
\theta(n+s_k).
\end{align}

Since \(c_r>s_k\), we have
\[
B_{\mathbf w}^{s_k}S_{\mathbf w}^{c_r}
=
S_{\mathbf w}^{c_r-s_k}.
\]
Moreover,
\[
\Phi_{n,q}^{p}(r)
=
P_n\bigl(\Phi_{n,q}^{p}(r)\bigr)
\qquad\text{and}\qquad
\bigl\|\Phi_{n,q}^{p}(r)\bigr\|_p<1.
\]
Therefore, by \eqref{eq-20-08-0850},
\begin{align} \label{eq-05-09-26-1702}
\begin{split}
\left\|
B_{\mathbf w}^{s_k}S_{\mathbf w}^{c_r}
\bigl(\Phi_{n,q}^{p}(r)\bigr)
\right\|_p^p
&=
\left\|
S_{\mathbf w}^{c_r-s_k}
P_n\bigl(\Phi_{n,q}^{p}(r)\bigr)
\right\|_p^p
\leq
\left\|
S_{\mathbf w}^{c_r-s_k}\circ P_n
\right\|^p
\\
&\leq
\sum_{t=1}^{n+1}
\frac{\|\mathbf w\|_\infty^{np}}
{\bigl|\omega_1\cdots
\omega_{c_r-s_k+t-1}\bigr|^p}
=
\sum_{\ell=c_r-s_k}^{d_r-s_k}
\frac{\|\mathbf w\|_\infty^{np}}
{\bigl|\omega_1\cdots\omega_\ell\bigr|^p},
\end{split}
\end{align}
where the last equality follows from \(d_r-c_r=n\).

For each fixed \(n\), the intervals
\[
\llbracket c_r-s_k,d_r-s_k\rrbracket,
\qquad
r\in A_{n,q},\quad r>r_k,
\]
are pairwise disjoint and, by \eqref{eq-05-09026-1656}, contained in
\(\llbracket\theta(n+s_k),\infty\rrbracket\). Therefore, using
\(\|\mathbf w\|_\infty\geq1\) and the defining property of
\(\theta\), we obtain 
\begin{align}
\begin{split}
\|R_q(s_k)\|_p^p
&=
\sum_{n\geq0}
\sum_{\substack{r\in A_{n,q}\\r>r_k}}
\left\|
B_{\mathbf w}^{s_k}S_{\mathbf w}^{c_r}
\bigl(\Phi_{n,q}^{p}(r)\bigr)
\right\|_p^p \\
&\hspace{-0.06in}\overset{\eqref{eq-05-09-26-1702}}{\leq}
\sum_{n\geq0}
\sum_{\ell\geq\theta(n+s_k)}
\frac{\|\mathbf w\|_\infty^{np}}
{\bigl|\omega_1\cdots\omega_\ell\bigr|^p}
\\
&\leq
\sum_{n\geq0}
\sum_{\ell\geq\theta(n+s_k)}
\frac{\|\mathbf w\|_\infty^{(n+s_k)p}}
{\bigl|\omega_1\cdots\omega_\ell\bigr|^p}
\\&<
\sum_{n\geq0}2^{-(n+s_k)}
=
2^{-s_k+1}.
\label{eq:7-sec4-remainder-estimate}
\end{split}
\end{align} 

\noindent\emph{Step 2: Choice of \(l_0\) and \(k_0\).}
Choose \(l_0>q_0\) such that
\begin{align} \label{eq-05-09-26-1837}
\sum_{q\geq l_0}|\lambda_q|^p
<
\frac{\delta^p}{2^{p+1}\beta^p}.
\end{align}
Since \(s_k\to\infty\) and there are only finitely many \(q<l_0\),
we may choose \(k_0\geq1\) such that, for every \(k\geq k_0\),
\begin{align}
s_k&>d_{2q}
\quad(0\leq q<l_0), \label{eq-05-09-26-1804}\\
\varepsilon 2^{(-s_k+1)/p} 
&<\frac{\delta}{4}, \label{eq-05-09-26-1805}\\
\varepsilon^p\|\lambda\|_p^p2^{-s_k+1}
&<\frac{\delta^p}{2^{p+1}}. \label{eq-05-09-26-1806}
\end{align}

\noindent\emph{Step 3: Estimate of \(E_1^{(4)}(k)\).}
For \(k\geq k_0\), the choice of \(y_k\),
\eqref{eq:7-sec4-decomposition},
\eqref{eq:7-sec4-remainder-estimate}, and \eqref{eq-05-09-26-1805} give
\begin{align*}
E_1^{(4)}(k)
&\leq
\|\varepsilon y_k-y_0\|_p
+
\varepsilon\|R_{q_0}(s_k)\|_p
\\
&<
\frac{\delta}{4}
+
\varepsilon2^{(-s_k+1)/p}
<
\frac{\delta}{2}.
\end{align*}  

\noindent\emph{Step 4: Estimate of \(E_2^{(4)}(k)\).}
By \eqref{eq-05-09-26-1804}, \[B_{\mathbf w}^{s_k}e_{d_{2q}}=0\] whenever \(q < l_0\). Now fix \(q\geq l_0\). Since \(2q\) is even and \(r_k\) is odd,
these indices are distinct. If \(2q<r_k\), then
\[
d_{2q}<c_{r_k}=s_k,
\]
so the corresponding vector again vanishes. If \(2q>r_k\), then
\(r_k\leq2q-1\), and the natural ordering gives
\[
s_k=c_{r_k}\leq d_{2q-1}.
\]
Hence, by \eqref{cond:4-even-b},
\[
c_{2q}
=
d_{2q}-d_{2q-1}+1
\leq
d_{2q}-s_k+1
\leq
d_{2q}.
\]
Applying \eqref{eq:4-even-products} with
\(u=d_{2q}-s_k+1\), we obtain
\[
\left\|
B_{\mathbf w}^{s_k}e_{d_{2q}}
\right\|_p
=
\prod_{\ell=d_{2q}-s_k+1}^{d_{2q}}
|\omega_\ell|
\leq\beta.
\]
Thus,
\[
E_2^{(4)}(k)
\leq
\beta^p\sum_{q\geq l_0}|\lambda_q|^p
\overset{\eqref{eq-05-09-26-1837}}{<}
\frac{\delta^p}{2^{p+1}}.
\]  

\noindent\emph{Step 5: Estimate of \(E_3^{(4)}(k)\).}
For \(q>q_0\), \eqref{eq:7-sec4-decomposition} gives
\[
B_{\mathbf w}^{s_k}\widetilde x_q=R_q(s_k).
\]
Therefore, by \eqref{eq:7-sec4-remainder-estimate},
\begin{align*}
E_3^{(4)}(k) \leq \varepsilon^p\|\lambda\|_p^p2^{-s_k+1}
\overset{\eqref{eq-05-09-26-1806}}{<} \frac{\delta^p}{2^{p+1}}.
\end{align*}
This proves Claim~\ref{claim-sec4-return-time-estimates}.

\subsection{Proof of Claim~\ref{claim:5-return-time-estimates}}
\label{Subsec-7.3}

Fix \(i\geq i_0\) and \(1\leq u'\leq N_{k_i}\), and set
\[
s':=\widetilde a_{k_i}(u').
\]
By \eqref{eq:5-selected-parameter},
\eqref{eq:5-selected-data}, \eqref{eq:5-target-approximation}, and
the definition of \(J_{k_i}\), we have
\begin{align}
(q_{k_i},n_{k_i},y_{k_i},m_{k_i},h_{k_i})
=(q_0,n_0,y_0,m_0,h_0), \quad 
\mu \in J_{k_i},
\quad
\sigma_{n_0+1}\leq\mu\leq r_{k_i}.
\label{eq:7.3-selected-data}
\end{align}

\smallskip
\noindent\emph{Step 1: Estimates for the even blocks \(\ell>k_i\).}
Fix such an \(\ell\). By the natural ordering of the coordinate
intervals and \eqref{cond:5-gap}, for \(1\leq v\leq N_\ell\),
\[
\frac{s'}{\widetilde a_\ell(v)}
\leq
\frac{\widetilde b_{\ell-1}(N_{\ell-1})}
     {\widetilde a_\ell(1)}
\leq\frac1\ell.
\]
Hence,
\begin{align} \label{eq-06-09-26-1428}
\frac{s'}{\widetilde a_\ell(v)-s'}
\leq\frac1{\ell-1}.
\end{align}
Using \(\log(1+x)\leq x\), it follows that, for
\(1\leq t\leq n_\ell+1\),
\begin{align} \label{eq-06-09-26-1823}
\begin{split}
\prod_{j=\widetilde a_\ell(v)+t-s'}^{
             \widetilde a_\ell(v)+t-1}
\left(1+\frac{\mu}{j}\right)
&\leq
\exp\left(
\mu\sum_{j=\widetilde a_\ell(v)+t-s'}^{
               \widetilde a_\ell(v)+t-1}\frac1j
\right)
\\&\leq
\exp\left(
\frac{\mu s'}{\widetilde a_\ell(v)-s'}
\right)
\\&\hspace{-0.07in}\overset{\eqref{eq-06-09-26-1428}}{\leq}
\exp\left(\frac{\mu}{\ell-1}\right)
\leq2,
\end{split}
\end{align}
where the last inequality follows from \eqref{eq:5-selected-large-k} and \(\ell > k_i\). By \eqref{eq:5-local-cover}, 
\begin{align} \label{eq-06-09-26-1830}
\sigma_{m_\ell}<r_\ell<m_\ell+1
\end{align}
and, then, by the
definition of \(\mathbf w_{r_\ell}\),
\begin{align} \label{eq-06-09-26-1831}
\|\mathbf w_{r_\ell}\|_\infty
=
\sup_{j\geq1}\left(1+\frac{r_\ell}{j}\right)
=
1+r_\ell
<
m_\ell+2.
\end{align}
Thus
\begin{align} \label{eq-06-09-26-1858}
\begin{split}
\sum_{v=1}^{N_\ell}
\left\|
B_{\mathbf w_\mu}^{s'}
S_{\mathbf w_{r_\ell}}^{\widetilde a_\ell(v)}y_\ell
\right\|_p^p
&\quad\leq
\sum_{v=1}^{N_\ell}\sum_{t=1}^{n_\ell+1}
\frac{\displaystyle
 \prod_{j=\widetilde a_\ell(v)+t-s'}^{
             \widetilde a_\ell(v)+t-1}
 \left(1+\frac{\mu}{j}\right)^p}
{\displaystyle
 \prod_{j=t}^{\widetilde a_\ell(v)+t-1}
 \left(1+\frac{r_\ell}{j}\right)^p}
\\
&\hspace{0.06in}\overset{\eqref{eq-06-09-26-1823}}{\leq}
\sum_{v=1}^{N_\ell}\sum_{t=1}^{n_\ell+1}
2^p
\frac{\displaystyle
 \prod_{j=1}^{t-1}
 \left(1+\frac{r_\ell}{j}\right)^p}
{\displaystyle
 \prod_{j=1}^{\widetilde a_\ell(v)+t-1}
 \left(1+\frac{r_\ell}{j}\right)^p}
\\
&\hspace{-0.11in}\overset{\eqref{eq-06-09-26-1830},\eqref{eq-06-09-26-1831}}{\leq}
\sum_{v=1}^{N_\ell}\sum_{t=1}^{n_\ell+1}
\frac{2^p(m_\ell+2)^{(t-1)p}}
{\displaystyle
 \prod_{j=1}^{\widetilde a_\ell(v)+t-1}
 \left(1+\frac{\sigma_{m_\ell}}{j}\right)^p}.
 \end{split}
\end{align} 
The integers
\[
\widetilde a_\ell(v)+t-1,
\qquad
1\leq v\leq N_\ell,\quad 1\leq t\leq n_\ell+1,
\]
are pairwise distinct. Indeed, this is immediate when \(v\) is fixed,
whereas, if \(v<v'\), then
\[
\bigl(\widetilde a_\ell(v')+t'-1\bigr)
-\bigl(\widetilde a_\ell(v)+t-1\bigr)
\overset{\eqref{cond:5-even-spacing}}{\geq}
\Theta_{n_\ell+1}-n_\ell
\geq1.
\]
Moreover, all these integers are at least
\(\widetilde a_\ell(1)\), and
\[
2^p(m_\ell+2)^{(t-1)p}
\leq[2(m_\ell+2)]^{(n_\ell+1)p}.
\]
Therefore, \eqref{cond:5-even-tail} and \eqref{eq-06-09-26-1858} yield
\begin{align}
\sum_{v=1}^{N_\ell}
\left\|
B_{\mathbf w_\mu}^{s'}
S_{\mathbf w_{r_\ell}}^{\widetilde a_\ell(v)}y_\ell
\right\|_p^p
<2^{-\ell}.
\label{eq:7.3-later-family-bound}
\end{align} 

\smallskip
\noindent\emph{Step 2: Estimates for the intervals
\(\widetilde I_{k_i}(v)\) with \(u'<v\leq N_{k_i}\).} By \eqref{eq:5-fixed-return-time}, \eqref{cond:5-even-spacing}, and
\eqref{eq:5-selected-data}, if \(u'<v\leq N_{k_i}\), then
\[
\widetilde a_{k_i}(v)-s'
=\widetilde a_{k_i}(v)-\widetilde a_{k_i}(u')
=(v-u')\Theta_{n_{k_i}+1}
=(v-u')\Theta_{n_0+1}.
\]
Hence, for \(1\leq t\leq n_0+1\), the inequalities
\(\mu\leq r_{k_i}\) and
\(r_{k_i}\geq\sigma_{n_0+1}\) give
\begin{align}
\begin{split} \label{eq-07-09-26-0051}
\frac{\displaystyle
 \prod_{j=\widetilde a_{k_i}(v)+t-s'}^{
             \widetilde a_{k_i}(v)+t-1}
 \left(1+\frac{\mu}{j}\right)^p}
{\displaystyle
 \prod_{j=t}^{\widetilde a_{k_i}(v)+t-1}
 \left(1+\frac{r_{k_i}}{j}\right)^p}
&\leq
\frac1{\displaystyle
 \prod_{j=t}^{\widetilde a_{k_i}(v)+t-s'-1}
 \left(1+\frac{r_{k_i}}{j}\right)^p}
\\
&\leq
\frac{\displaystyle
 \prod_{j=1}^{t-1}
 \left(1+\frac{\sigma_{n_0+1}}{j}\right)^p}
{\displaystyle
 \prod_{j=1}^{(v-u')\Theta_{n_0+1}+t-1}
 \left(1+\frac{\sigma_{n_0+1}}{j}\right)^p}
\\
&\leq
\frac{\|\mathbf w_{\sigma_{n_0+1}}\|_\infty^{(n_0+1)p}}
{\displaystyle
 \prod_{j=1}^{(v-u')\Theta_{n_0+1}+t-1}
 \left(1+\frac{\sigma_{n_0+1}}{j}\right)^p}.
 \end{split}
\end{align}
The integers
\[
(v-u')\Theta_{n_0+1}+t-1,
\qquad
u'<v\leq N_{k_i},\quad 1\leq t\leq n_0+1,
\]
are pairwise distinct and at least \(\Theta_{n_0+1}\). Indeed, if
\(v<v'\), then the difference of the corresponding integers is at
least
\[
\Theta_{n_0+1}-n_0\geq1,
\]
while the case \(v=v'\) is immediate. Using \(\|y_0\|_p<1\), \eqref{eq:5-within-family-tail} and \eqref{eq-07-09-26-0051}, we have
\begin{align}
\begin{split} 
\sum_{v=u'+1}^{N_{k_i}}
\left\|
B_{\mathbf w_\mu}^{s'}
S_{\mathbf w_{r_{k_i}}}^{\widetilde a_{k_i}(v)}y_0
\right\|_p^p
&\leq
\sum_{v=u'+1}^{N_{k_i}}\sum_{t=1}^{n_0+1}
\frac{\|\mathbf w_{\sigma_{n_0+1}}\|_\infty^{(n_0+1)p}}
{\displaystyle
 \prod_{j=1}^{(v-u')\Theta_{n_0+1}+t-1}
 \left(1+\frac{\sigma_{n_0+1}}{j}\right)^p}
\\
&\leq
\sum_{d\geq\Theta_{n_0+1}}
\frac{\|\mathbf w_{\sigma_{n_0+1}}\|_\infty^{(n_0+1)p}}
{\displaystyle
 \prod_{j=1}^{d}
 \left(1+\frac{\sigma_{n_0+1}}{j}\right)^p}
<2^{-(n_0+1)}.
\label{eq:7.3-same-family-bound}
\end{split}
\end{align}

\smallskip
\noindent\emph{Step 3: Residual estimates for \(E_1^{(5)}(s')\) and \(E_3^{(5)}(s')\).}
By \eqref{eq:5-selected-data}, \(q_{k_i}=q_0\), so the definition
of \(\widetilde x_{q_0}\) contains the summand
\[
S_{\mathbf w_{r_{k_i}}}^{\widetilde a_{k_i}(u')}y_{k_i}.
\]
Since \(y_{k_i}=y_0\) by \eqref{eq:5-selected-data} and
\(\widetilde a_{k_i}(u')=s'\) by
\eqref{eq:5-fixed-return-time}, this summand is
\[
S_{\mathbf w_{r_{k_i}}}^{\widetilde a_{k_i}(u')}y_{k_i}
=
S_{\mathbf w_{r_{k_i}}}^{s'}y_0.
\]
Every preceding summand is annihilated by
\(B_{\mathbf w_\mu}^{s'}\), while the remaining summands have
pairwise disjoint supports. Hence
\begin{align}
\begin{split}
\left\|
B_{\mathbf w_\mu}^{s'}
\left(
\widetilde x_{q_0}
-S_{\mathbf w_{r_{k_i}}}^{s'}y_0
\right)
\right\|_p^p
&=
\sum_{v=u'+1}^{N_{k_i}}
\left\|
B_{\mathbf w_\mu}^{s'}
S_{\mathbf w_{r_{k_i}}}^{\widetilde a_{k_i}(v)}y_0
\right\|_p^p
+
\sum_{\substack{\ell>k_i,\ \ell\ {\rm even}\\ q_\ell=q_0}}
\sum_{v=1}^{N_\ell}
\left\|
B_{\mathbf w_\mu}^{s'}
S_{\mathbf w_{r_\ell}}^{\widetilde a_\ell(v)}y_\ell
\right\|_p^p\\
&\hspace{-0.25in}\overset{
\eqref{eq:7.3-same-family-bound},
\eqref{eq:7.3-later-family-bound}
}{<}
2^{-(n_0+1)}
+
\sum_{\substack{\ell>k_i,\ \ell\ {\rm even}\\ q_\ell=q_0}}
2^{-\ell}
\\
&\leq
2^{-(n_0+1)}+\sum_{\ell>k_i}2^{-\ell}
\overset{\eqref{eq:5-selected-large-k}}{<}\eta_0.
\label{eq:7.3-residual-q0}
\end{split}
\end{align}

If \(q\neq q_0\), then \(q_{k_i}=q_0\) shows that no summand with
index \(k_i\) occurs in \(\widetilde x_q\). Thus, the same argument
gives
\begin{align}
\left\|
B_{\mathbf w_\mu}^{s'}\widetilde x_q
\right\|_p^p
&<
\sum_{\ell>k_i}2^{-\ell}
<\eta_0.
\label{eq:7.3-residual-other-q}
\end{align}

\smallskip
\noindent\emph{Step 4: The three errors.}
By \eqref{eq:5-parameter-control}, \eqref{eq:5-eta0}, 
\eqref{eq:5-target-approximation}, \eqref{eq:5-h-m-choice}, and
\eqref{eq:7.3-residual-q0},
\begin{align*}
E_1^{(5)}(s')
&\leq
\varepsilon
\left\|
B_{\mathbf w_\mu}^{s'}
S_{\mathbf w_{r_{k_i}}}^{s'}y_0-y_0
\right\|_p
+\|\varepsilon y_0-z_0\|_p
+
\varepsilon
\left\|
B_{\mathbf w_\mu}^{s'}
\left(
\widetilde x_{q_0}
-S_{\mathbf w_{r_{k_i}}}^{s'}y_0
\right)
\right\|_p
\\
&<
\varepsilon2^{-h_0}
+\frac{\delta}{4}
+\varepsilon\eta_0^{1/p}
<
\frac{\delta}{8}
+\frac{\delta}{4}
+\frac{\delta}{8}
=\frac{\delta}{2}.
\end{align*}

For \(0\leq q<Q\), the first condition in
\eqref{eq:5-selected-large-k} gives
\[
\widetilde b_{2q+1}(1)
<\widetilde a_{k_i}(1)
\leq s',
\]
and hence
\[
B_{\mathbf w_\mu}^{s'}e_{\widetilde b_{2q+1}(1)}=0.
\]
Now let \(q\geq Q\). The coordinates
\(\widetilde b_{2q+1}(1)\) and \(s'\) cannot be equal, because they
belong to two distinct coordinate intervals. If
\(\widetilde b_{2q+1}(1)<s'\), the corresponding vector is again
annihilated. If \(\widetilde b_{2q+1}(1)>s'\), the natural ordering
of the coordinate intervals gives
\[
s'\leq\widetilde b_{2q}(N_{2q}).
\]
Moreover, by \eqref{eq:5-choice-Q},
\[
1<\mu\leq2Q+1\leq2q+1.
\]
Therefore, \eqref{eq:5-odd-product-control}, applied with the odd
integer \(2q+1\) and \(j=s'\), yields
\[
\left\|
B_{\mathbf w_\mu}^{s'}e_{\widetilde b_{2q+1}(1)}
\right\|_p^p
\leq2.
\]
It follows from \eqref{eq:5-choice-Q} that
\[
E_2^{(5)}(s')
\leq
2\sum_{q\geq Q}|\lambda_q|^p
<\frac{\delta^p}{2^{p+1}}.
\]

Finally, \eqref{eq:7.3-residual-other-q} and
\eqref{eq:5-eta0} give
\begin{align*}
E_3^{(5)}(s')
\leq
\varepsilon^p\eta_0
\sum_{q>q_0}|\lambda_q|^p
<
\varepsilon^p\eta_0\|\lambda\|_p^p
\leq
\frac{\delta^p}{2^{p+1}},
\end{align*}
where the strict inequality follows from \(\lambda_{q_0}=1\).
This proves Claim~\ref{claim:5-return-time-estimates}.

\section*{\centering Acknowledgments}

The first author thanks the Instituto de Matem\'atica e Estat\'istica of the Universidade Federal de Uberl\^andia for its hospitality during his two-month visit. His research was partially supported by CNPq under grants 312167/2021-0, 403964/2024-5, and 406457/2023-9. The second author was supported by FAPEMIG under grant APQ-02622-26. The third and fourth authors were supported by FAPEMIG under grants RED-00133-21, APQ-01853-23, and APQ-02622-26. The authors are also indebted to an anonymous referee of an earlier version of the manuscript, whose suggestions led to a substantial rewriting of the paper, with new notation and simpler arguments. The exposition has benefited greatly from these suggestions.

\section*{\centering Declaration of LLM Use}

All results, proofs, and constructions in the first version of the manuscript were developed by the authors without AI assistance, as was its writing. During a subsequent revision, they identified a gap in the proof of the second main theorem (Theorem \ref{thm-2}). Only then did the authors engage in iterative exchanges with GPT-5.6, specifically to help develop and check a correction of that proof.

The gap concerned the passage from the countably many parameter values treated by the recursive construction to the full uncountable family of weighted shifts. The correction emerged through these exchanges: it retained the basic strategy but modified the construction by introducing a countable cover of \(\mathbb{R}_{>1}\) by parameter intervals on which the required estimates held uniformly. In the process of verifying that this modification worked, the authors needed to formulate Lemmas \ref{lem:exponential-tail} and \ref{lem:finite-completion}, and GPT-5.6 provided substantial assistance in finding their proofs.

In addition, some of the authors used  AI tools for English-language editing of the manuscript, limited to grammar and phrasing. This language editing took place at the final stage of preparation, after the mathematical work had been completed and before submission.

The authors determined the final form of the proofs and take full responsibility for the content of the article.


\begin{thebibliography}{30}

\bibitem{BayErnMen16}
\textsc{F. Bayart, R. Ernst and Q. Menet},
\textit{Non-existence of frequently hypercyclic subspaces for $P(D)$},
Israel J. Math. \textbf{214} (2016), no.~1, 149--166.
\href{https://doi.org/10.1007/s11856-016-1352-3}
{doi:10.1007/s11856-016-1352-3}.


\bibitem{BayGri2004}
\textsc{F. Bayart and S. Grivaux},
\textit{Hypercyclicit\'e : le r\^ole du spectre ponctuel unimodulaire},
C. R. Math. Acad. Sci. Paris \textbf{338} (2004), no.~9, 703--708.
\href{https://doi.org/10.1016/j.crma.2004.02.012}
{doi:10.1016/j.crma.2004.02.012}.


\bibitem{BayGri2006}
\textsc{F. Bayart and S. Grivaux},
\textit{Frequently hypercyclic operators},
Trans. Amer. Math. Soc. \textbf{358} (2006), no.~11, 5083--5117.
\href{https://doi.org/10.1090/S0002-9947-06-04019-0}
{doi:10.1090/S0002-9947-06-04019-0}.


\bibitem{BayMath-book}
\textsc{F. Bayart and E. Matheron},
\textit{Dynamics of Linear Operators},
Cambridge Tracts in Mathematics, vol.~179,
Cambridge University Press, Cambridge, 2009.
\href{https://doi.org/10.1017/CBO9780511581113}
{doi:10.1017/CBO9780511581113}.


\bibitem{BayRuz2015}
\textsc{F. Bayart and I. Z. Ruzsa},
\textit{Difference sets and frequently hypercyclic weighted shifts},
Ergodic Theory Dynam. Systems \textbf{35} (2015), no.~3, 691--709.
\href{https://doi.org/10.1017/etds.2013.77}
{doi:10.1017/etds.2013.77}.


\bibitem{Bernal06}
\textsc{L. Bernal-Gonz\'alez},
\textit{Hypercyclic subspaces in Fr\'echet spaces},
Proc. Amer. Math. Soc. \textbf{134} (2006), no.~7, 1955--1961.
\href{https://doi.org/10.1090/S0002-9939-05-08242-0}
{doi:10.1090/S0002-9939-05-08242-0}.


\bibitem{BerCalLopPra25}
\textsc{L. Bernal-Gonz\'alez, M. C. Calder\'on-Moreno,
	J. L\'opez-Salazar and J. A. Prado-Bassas},
\textit{Hypercyclic subspaces for sequences of finite order differential operators},
J. Math. Anal. Appl. \textbf{546} (2025), no.~1, Paper No.~129257.
\href{https://doi.org/10.1016/j.jmaa.2025.129257}
{doi:10.1016/j.jmaa.2025.129257}.


\bibitem{BerMon1995}
\textsc{L. Bernal-Gonz\'alez and A. Montes-Rodr\'iguez},
\textit{Universal functions for composition operators},
Complex Variables Theory Appl. \textbf{27} (1995), no.~1, 47--56.
\href{https://doi.org/10.1080/17476939508814804}
{doi:10.1080/17476939508814804}.


\bibitem{BesMenet15}
\textsc{J. B\`es and Q. Menet},
\textit{Existence of common and upper frequently hypercyclic subspaces},
J. Math. Anal. Appl. \textbf{432} (2015), no.~1, 10--37.
\href{https://doi.org/10.1016/j.jmaa.2015.06.007}
{doi:10.1016/j.jmaa.2015.06.007}.


\bibitem{BonillaErdmann12}
\textsc{A. Bonilla and K.-G. Grosse-Erdmann},
\textit{Frequently hypercyclic subspaces},
Monatsh. Math. \textbf{168} (2012), no.~3--4, 305--320.
\href{https://doi.org/10.1007/s00605-011-0369-2}
{doi:10.1007/s00605-011-0369-2}.


\bibitem{ChanMad25}
\textsc{K. C. Chan and Z. Madarasz},
\textit{A strictly weakly hypercyclic operator with a hypercyclic subspace},
J. Operator Theory \textbf{93} (2025), no.~2, 435--476.
\href{https://doi.org/10.7900/JOT.2023JUN06.2433}
{doi:10.7900/JOT.2023JUN06.2433}.


\bibitem{DasMun24}
\textsc{B. K. Das and A. Mundayadan},
\textit{Dynamics of weighted backward shifts on certain analytic function spaces},
Results Math. \textbf{79} (2024), Paper No.~242.
\href{https://doi.org/10.1007/s00025-024-02279-0}
{doi:10.1007/s00025-024-02279-0}.


\bibitem{GonLeonMont-2000}
\textsc{M. Gonz\'alez, F. Le\'on-Saavedra and A. Montes-Rodr\'iguez},
\textit{Semi-Fredholm theory: hypercyclic and supercyclic subspaces},
Proc. Lond. Math. Soc. (3) \textbf{81} (2000), no.~1, 169--189.
\href{https://doi.org/10.1112/S0024611500012454}
{doi:10.1112/S0024611500012454}.


\bibitem{Gro-ErdMang2011-book}
\textsc{K.-G. Grosse-Erdmann and A. Peris Manguillot},
\textit{Linear Chaos},
Universitext, Springer, London, 2011.
\href{https://doi.org/10.1007/978-1-4471-2170-1}
{doi:10.1007/978-1-4471-2170-1}.


\bibitem{guerre}
\textsc{S. Guerre-Delabri\`ere},
\textit{Classical Sequences in Banach Spaces},
Monographs and Textbooks in Pure and Applied Mathematics, vol.~166,
Marcel Dekker, New York, 1992.


\bibitem{LeonMon97}
\textsc{F. Le\'on-Saavedra and A. Montes-Rodr\'iguez},
\textit{Linear structure of hypercyclic vectors},
J. Funct. Anal. \textbf{148} (1997), no.~2, 524--545.
\href{https://doi.org/10.1006/jfan.1996.3084}
{doi:10.1006/jfan.1996.3084}.


\bibitem{LeonMul06}
\textsc{F. Le\'on-Saavedra and V. M\"uller},
\textit{Hypercyclic sequences of operators},
Studia Math. \textbf{175} (2006), no.~1, 1--18.
\href{https://doi.org/10.4064/sm175-1-1}
{doi:10.4064/sm175-1-1}.


\bibitem{Lopez24}
\textsc{A. L\'opez-Mart\'inez},
\textit{Recurrent subspaces in Banach spaces},
Int. Math. Res. Not. IMRN \textbf{2024} (2024), no.~11, 9067--9087.
\href{https://doi.org/10.1093/imrn/rnad321}
{doi:10.1093/imrn/rnad321}.


\bibitem{Menet2014}
\textsc{Q. Menet},
\textit{Hypercyclic subspaces and weighted shifts},
Adv. Math. \textbf{255} (2014), 305--337.
\href{https://doi.org/10.1016/j.aim.2014.01.012}
{doi:10.1016/j.aim.2014.01.012}.


\bibitem{Menet2015}
\textsc{Q. Menet},
\textit{Existence and non-existence of frequently hypercyclic subspaces for weighted shifts},
Proc. Amer. Math. Soc. \textbf{143} (2015), no.~6, 2469--2477.
\href{https://doi.org/10.1090/S0002-9939-2015-12444-6}
{doi:10.1090/S0002-9939-2015-12444-6}.


\bibitem{Montes96}
\textsc{A. Montes-Rodr\'iguez},
\textit{Banach spaces of hypercyclic vectors},
Michigan Math. J. \textbf{43} (1996), no.~3, 419--436.
\href{https://doi.org/10.1307/mmj/1029005536}
{doi:10.1307/mmj/1029005536}.

\bibitem{Rudin}
\textsc{W. Rudin}, \textit{Principles of Mathematical Analysis}, 3rd ed., McGraw-Hill, New York, 1976.


\bibitem{Shkarin09}
\textsc{S. Shkarin},
\textit{On the spectrum of frequently hypercyclic operators},
Proc. Amer. Math. Soc. \textbf{137} (2009), no.~1, 123--134.
\href{https://doi.org/10.1090/S0002-9939-08-09655-X}
{doi:10.1090/S0002-9939-08-09655-X}.


\bibitem{Shkarin10}
\textsc{S. Shkarin},
\textit{On the set of hypercyclic vectors for the differentiation operator},
Israel J. Math. \textbf{180} (2010), 271--283.
\href{https://doi.org/10.1007/s11856-010-0104-z}
{doi:10.1007/s11856-010-0104-z}.


\end{thebibliography}
\end{document}